\documentclass[a4]{amsart}
\usepackage{amsmath,amsfonts,amsthm,amsopn,amssymb}
\usepackage{cite,marginnote}
\usepackage{color,xcolor,enumitem,graphicx}
\usepackage[colorlinks=true,urlcolor=blue,
citecolor=red,linkcolor=blue,linktocpage,pdfpagelabels,
bookmarksnumbered,bookmarksopen]{hyperref}
\usepackage[english]{babel}

\usepackage[left=2.9cm,right=2.9cm,top=2.8cm,bottom=2.8cm]{geometry}

\def\sideremark#1{\ifvmode\leavevmode\fi\vadjust{\vbox to0pt{\vss
 \hbox to 0pt{\hskip\hsize\hskip1em
 \vbox{\hsize2.1cm\tiny\raggedright\pretolerance10000
  \noindent #1\hfill}\hss}\vbox to15pt{\vfil}\vss}}}%

\usepackage[hyperpageref]{backref}

\numberwithin{equation}{section}

\makeindex
\makeindex

\newtheorem{thm}{Theorem}[section]
\newtheorem{lemma}[thm]{Lemma}
\newtheorem{prop}[thm]{Proposition}
\newtheorem{cor}[thm]{Corollary}
\theoremstyle{definition}
\newtheorem{remark}[thm]{Remark}
\newtheorem{defn}{Definition}[section]

\newcommand{\s}{\section}

\newcommand{\R}{\mathbb R}

\newcommand{\bt}{\begin{theorem}}
\newcommand{\et}{\end{theorem}}
\newcommand{\bl}{\begin{lemma}}
\newcommand{\el}{\end{lemma}}
\newcommand{\bd}{\begin{definition}}
\newcommand{\ed}{\end{definition}}
\newcommand{\bc}{\begin{corollary}}
\newcommand{\ec}{\end{corollary}}
\newcommand{\bp}{\begin{proof}}
\newcommand{\ep}{\end{proof}}
\newcommand{\bx}{\begin{example}}
\newcommand{\ex}{\end{example}}
\newcommand{\bi}{\begin{exercise}}
\newcommand{\ei}{\end{exercise}}
\newcommand{\bo}{\begin{proposition}}
\newcommand{\eo}{\end{proposition}}
\newcommand{\br}{\begin{remark}}
\newcommand{\er}{\end{remark}}
\newcommand{\beq}{\begin{equation}}
\newcommand{\eeq}{\end{equation}}
\newcommand{\ba}{\begin{align}}
\newcommand{\ea}{\end{align}}
\newcommand{\bn}{\begin{enumerate}}
\newcommand{\en}{\end{enumerate}}
\newcommand{\bg}{\begin{align*}}
\newcommand{\bcs}{\begin{cases}}
\newcommand{\ecs}{\end{cases}}

\newcommand{\bean}{\begin{eqnarray*}}
\newcommand{\eean}{\end{eqnarray*}}

\def\R{\mathbb{R}}

\def\bd{\mathrm{bd}\,}

\title[Existence and nonexistence]{Existence and nonexistence of least energy positive solutions to critical Schr\"{o}dinger systems with Hardy potential}

\author[S.~You]{Song You}
\author[J.~J.~Zhang]{Jianjun Zhang}

\address[S.~You]{\newline\indent School of Mathematical Sciences
\newline\indent
Chongqing Normal University
\newline\indent
Chongqing, 401331, PR China}
\email{\href{mailto:yousong@cqnu.edu.cn}{yousong@cqnu.edu.cn}}

\address[J.~J.~Zhang]{\newline\indent College of Mathematics and Statistics
\newline\indent
Chongqing Jiaotong University
\newline\indent
Chongqing 400074, PR China}
\email{\href{mailto:zhangjianjun09@tsinghua.org.cn}{zhangjianjun09@tsinghua.org.cn}}


\begin{document}

\begin{abstract}
We are concerned with the following coupled Schr\"{o}dinger system with Hardy potential in the critical case
  \begin{equation*}
  \begin{cases}
  -\Delta u_{i}-\frac{\lambda_{i}}{|x|^2}u_{i}=|u_i|^{2^*-2}u_i+\sum_{j\neq i}^{3}\beta_{ij}|u_{j}|^{\frac{2^*}{2}}|u_i|^{\frac{2^*}{2}-2}u_i, ~x\in \mathbb{R}^N,\\
  u_i\in D^{1,2}(\mathbb{R}^N),\,\, N\geq 3,\,\, i=1,2,3,
  \end{cases}
  \end{equation*}
  where $2^*=\frac{2N}{N-2}$, $\lambda_i\in (0,\Lambda_N), \Lambda_N:= \frac{(N-2)^2}{4}$, $\beta_{ij}=\beta_{ji}$ for $i \neq j$. By virtue of variational methods, we establish  the existence and nonexistence of least energy solutions for the purely cooperative case ($\beta_{ij}> 0$ for any $i\neq j$) and the simultaneous cooperation and competition case ($\beta_{i_{1}j_{1}}>0$ and $\beta_{i_{2}j_{2}}<0$ for some $(i_{1}, j_{1})$ and $(i_{2}, j_{2})$).  Moreover, it is shown that fully nontrivial ground state solutions exist when $\beta_{ij}\ge0$ and $N\ge5$, but {\bf NOT} in the weakly pure cooperative case ($\beta_{ij}>0$ and small, $i\neq j$) when $N=3,4$. We emphasize that this reveals that the existence of ground state solutions differs dramatically between $N=3, 4$ and higher dimensions $N\geq 5$. In particular, the cases of $N=3$ and $N\geq 5$ are more complicated than the case of $N=4$ and the proofs heavily depend on the dimension. Some novel tricks are introduced for $N=3$ and $N\ge5$.

\end{abstract}

\maketitle

\s{Introduction}
\renewcommand{\theequation}{1.\arabic{equation}}

\subsection{Background}
Consider the following Schr\"{o}dinger system
\begin{equation}\label{system1}
\begin{cases}
\imath \partial_t \Psi_i +\Delta \Psi_i-a_i(x)\Psi_i +  \sum_{j=1}^{d}\beta_{ij}|\Psi_j|^{p}|\Psi_i|^{p-2}\Psi_i=0, \\
\Psi_i=\Psi_i(x,t),  ~i=1,\ldots,d,~x\in \mathbb{R}^N, ~t>0,\\
\Psi_i(x,t)\rightarrow 0, \text{ as } |x|\rightarrow +\infty, ~i=1,\ldots,d,
\end{cases}
\end{equation}
where $\imath$ is the imaginary unit, $a_i(x)$ are potential functions, $\beta_{ii}>0, i=1,\ldots,d$, $\beta_{ij}=\beta_{ji}$ for $i \neq j$. System \eqref{system1} comes from many physical models; for example, it can be used to explain the coexistence of several different Bose-Einstein condensates (see \cite{Timmermans1998}). In quantum mechanics, the solutions $\Psi_i(i=1,...3)$ are the corresponding wave amplitudes, $\beta_{ii}$ represent self-interactions within the same component, while $\beta_{ij}$ $(i\neq j)$ describe the strength and type of interactions between different components $\Psi_i$ and $\Psi_j$. Furthermore, $\beta_{ij}>0$ represents a cooperative interaction, while, when $\beta_{ij}<0$, the interaction is competitive.

In order to obtain standing wave solutions of system \eqref{system1}, set $\Psi_i(x,t)=e^{\imath \lambda_{i}t}u_i(x)$. System \eqref{system1} is reduced to the following elliptic system
\begin{equation}\label{system2}
\begin{cases}
-\Delta u_i+V_{i}(x)u_i=\sum\limits_{j=1}^{d} \beta_{ij}|u_j|^{p}|u_i|^{p-2}u_i, ~i=1,\ldots,d, ~x\in \mathbb{R}^N,\\
u_i(x)\rightarrow 0, \text{ as } |x|\rightarrow +\infty, ~i=1,\ldots,d,
\end{cases}
\end{equation}
where $V_i(x)=a_i(x)+\lambda_i$, $d\geq 2$. For the subcritical case $1<p< 2^*/2$, $2^*=\frac{2N}{N-2}$, $N\geq 3$, and constant positive potentials $V_{i}(x)=\lambda_i>0$ for any $i=1,\ldots, d$, the existence of solutions to \eqref{system2} has been studied intensively in the last two decades with $d=2$, see \cite{Ambrosetti2007,Bartsch-Dancer-Wang2010,MaiaMontefuscoPellacci,Mandel2015,Sirakov2007,Terracini-Verzini2009,WT2008} and references therein. In this situation, there is only one interaction constant $\beta=\beta_{12}=\beta_{21}$. For an arbitrary number of equations $d\geq 3$, the problem becomes more complicated due to the possible occurrence of simultaneous cooperation and competition, that is, the existence of at least two pairs, $(i_1,j_1)$ and $(i_2,j_2)$, such that $i_1\neq j_1, i_2\neq j_2$, $\beta_{i_1 j_1}>0$ and $\beta_{i_2 j_2}<0$; the existence of solutions to \eqref{system2} has attracted a lot of attention, see for example \cite{BSWang2016,BLWang2019,Correia-Oliveria-Tavares2016,Lin-Wei2005,Peng-Wang2019,Sato-Wang2015,Soave-Tavares2016,Soave2015} and references therein.

\subsection{Motivation}
For the critical case $2p=2^*$ and $N\geq 3$, if the potentials are constant $V_{i}(x)=\lambda_i$ for any $i=1,\ldots, d$, by letting $\mathbf{u}=(u_1,\ldots,u_d)$ be any solution of \eqref{system2},  ones deduce from Pohozaev identity that $\int_{\mathbb{R}^N}\sum_{i=1}^d \lambda_i u_i^2=0$, which yields that, if $\lambda_1,\ldots, \lambda_d\neq 0$, then $u_1,\ldots, u_d\equiv 0$ for any $i=1,\ldots, d$. Next, we consider the case where $V_{i}(x)$ is a nonconstant function for any $i=1,\ldots, d$; in particular, consider the case when $V_i(x)=-\frac{\lambda_{i}}{|x|^2}$ for any $i=1,\ldots, d$. When $V_i(x)=-\frac{\lambda_{i}}{|x|^2}$, $d=2$, $2p=2^*$ and $N\geq 3$, in \cite{Abdellaoui2009} the authors proved the existence and nonexistence of ground state solution to \eqref{system2} with more general powers by using the concentration-compactness principle. Subsequently, Chen and Zou \cite{Chen-Zou2015} obtained the existence of least energy positive solutions to \eqref{system2} with $V_i(x)=-\frac{\lambda_{i}}{|x|^2}$ and more general powers by complicated analysis.

To the best of our knowledge, less considered for system \eqref{system2} is the critical case ($2p=2^*$) with $d\geq3$ and $V_i(x)=-\frac{\lambda_{i}}{|x|^2}$, $i=1,\ldots, d$. In this aspect, we refer to \cite{Guo-Luo-Zou2021} for ground state solutions of \eqref{system2}. A natural question is
\begin{center}
{\it whether system \eqref{system2} with $V_i(x)=-\frac{\lambda_{i}}{|x|^2}$ admits a least energy positive solution}
\end{center}
 when $d\geq3$. Here we refer to \eqref{leastenergylevel} for the definition of least energy positive solutions and point out that ground state solutions can be semi-trivial. The main purpose of the present paper is to make a contribution in this direction. Precisely, in what follows, we consider the coupled elliptic systems as follows
\begin{equation}\label{mainsystem}
\begin{cases}
-\Delta u_{i}-\frac{\lambda_{i}}{|x|^2}u_{i}=|u_i|^{2^*-2}u_i+\sum_{j\neq i}^{3}\beta_{ij}|u_{j}|^{\frac{2^*}{2}}|u_i|^{\frac{2^*}{2}-2}u_i, ~x\in \mathbb{R}^N, ~i=1,2,3,\\
u_i\in D^{1,2}(\mathbb{R}^N), i=1,2,3,
\end{cases}
\end{equation}
where $2^*=\frac{2N}{N-2}$, $N\geq3$, $\lambda_i\in (0,\Lambda_N), \Lambda_N:= \frac{(N-2)^2}{4}$, $\beta_{ii}=1, i=1,2,3$, $\beta_{ij}=\beta_{ji}$ for $i \neq j$, and $D^{1,2}(\mathbb{R}^N):= \{u\in L^{2^*}(\mathbb{R}^N): |\nabla u|\in L^2(\mathbb{R}^N)\}$ with norm
\begin{equation*}
\|u\|^2:=\int_{\mathbb{R}^N}|\nabla u|^2.
\end{equation*}
By Hardy's inequality,
\begin{equation*}
\Lambda_N \int_{\mathbb{R}^N}\frac{v^{2}}{|x|^2} \leq  \int_{\mathbb{R}^N}|\nabla v|^2, \quad \forall v\in D^{1,2}(\mathbb{R}^N),
\end{equation*}
we have that, for $\lambda_i\in (0,\Lambda_N)$,
\begin{equation*}
\|u\|_i^2:=\int_{\mathbb{R}^N}\left(|\nabla u|^2-\frac{\lambda_{i}}{|x|^2} u^{2}\right)
\end{equation*}
are equivalent norms to $\|\cdot\|$, for $i=1,2,3$. By \cite{Terracini1996}, the following problem
\begin{equation}\label{scalarequation}
\begin{cases}
-\Delta v-\frac{\lambda_{i}}{|x|^2}v=v^{2^*-1}, ~ x\in \R^N,\\
v\in D^{1,2}(\mathbb{R}^N), ~ v>0 \text{ in } \R^N\backslash \{0\}
\end{cases}
\end{equation}
has exactly a one-dimensional $C^2$ positive solutions given by
\begin{equation}\label{tem512-3}
\mathcal{Z}_i=\left\{ z^i_\mu(x)=\mu^{-\frac{N-2}{2}} z^i_1(\frac{x}{\mu}), \mu>0 \right\},
~\text{ where }
z^i_1(x)=\frac{C}{|x|^a\left( 1+ |x|^{2-\frac{4a}{N-2}}\right)^{\frac{N-2}{2}}},
\end{equation}
$a=\frac{N-2}{2}-\sqrt{\frac{(N-2)^2}{4}-\lambda_i}$ and $C>0$ is a constant.
Moreover, all positive solutions of \eqref{scalarequation} satisfy
\begin{equation}\label{tem529-8-4}
S_i:= \inf_{v\in D^{1,2}(\mathbb{R}^N)\backslash \{0\}} \frac{\|v\|_i^2}{|v|^2_{2^*}}=\frac{\|\omega_i\|_i^2}{|\omega_i|^2_{2^*}}=\left(1-\frac{4\lambda_i}{(N-2)^2}
\right)^{\frac{N-1}{N}}\mathcal{S},
\end{equation}
and
\begin{equation}\label{tem529-1}
A_i=I_i(\omega_i):=\frac{1}{2}\|\omega_i\|_i^2-\frac{1}{2^*}|\omega_i|_{2^*}^{2^*}
=\frac{1}{N}S_i^{\frac{N}{2}}<\frac{1}{N}\mathcal{S}^{\frac{N}{2}},
\end{equation}
where $\mathcal{S}$ is the sharp constant of the embedding $D^{1,2}(\mathbb{R}^N)\hookrightarrow L^{2^*}(\mathbb{R}^N)$,
\begin{equation*}
\mathcal{S}\left( \int_{\mathbb{R}^N} |v|^{2^*}\right)^{\frac{2}{2^*}}\leq \int_{\mathbb{R}^N} |\nabla v|^2.
\end{equation*}
Moreover, we have
\begin{equation}\label{tem127-4}
\int_{\mathbb{R}^N}\left(|\nabla u|^{2}-\frac{\lambda_{i}}{|x|^2} u^{2}\right)\geq
(NA_i)^{\frac{2}{N}}\left(\int_{\mathbb{R}^N}|u|^{2^*}\right)^{\frac{2}{2^*}}, ~\forall u\in D^{1,2}(\mathbb{R}^N).
\end{equation}

Let $\mathbf{D}=D^{1,2}(\R^N;\R^3)$. It is well known that any solution of \eqref{mainsystem} corresponds to a critical point of the $C^{1}$ functional $I:\mathbf{D} \rightarrow \mathbb{R}$ given by
\begin{align}\label{1.10}
I(\mathbf{u}) & =\frac{1}{2}\sum_{i=1}^3\int_{\mathbb{R}^N}\left(|\nabla u|^{2}-\frac{\lambda_{i}}{|x|^2} u_i^{2}\right)-\frac{1}{2^*}\sum_{i=1}^3\int_{\mathbb{R}^N}|u_i|^{2^*}
      -\frac{1}{2^*}\sum_{i\neq j}^3\int_{\mathbb{R}^N}\beta_{ij} |u_i|^{\frac{2^*}{2}}|u_j|^{\frac{2^*}{2}}.
\end{align}

\begin{defn}
A solution $\mathbf{u}=(u_1,u_2,u_3)$ of \eqref{mainsystem} is called {\it trivial } if  all its components are zero, i.e., $u_1,u_2,u_3\equiv 0$. A solution is called {\it semi-trivial } if there exist at least one (but not all) vanishing component. A solution $\mathbf{u}$ is called {\it fully nontrivial} if all of its components $u_1,u_2,u_3$ are nontrivial. A solution $\mathbf{u}$ is {\it positive} (resp. {\it nonnegative}), if $u_i>0$ (resp. $u_i\geq 0$) for every $i=1,2,3$.
\end{defn}

In this paper, we focus on the existence of {\it ground state solutions}, achieving the following ground state level
\begin{equation}\label{eqgroundstatelevel}
\inf\{I(\mathbf{u}):\  I'(\mathbf{u})= 0,\ \mathbf{u}\in \mathbf{D}, \mathbf{u}\neq \mathbf{0} \}
\end{equation}
and
the existence of  {\it least energy positive solutions,} which attain the  least energy positive level
\begin{equation}\label{leastenergylevel}
\inf\left\lbrace J(\mathbf{u}): I'(\mathbf{u})= 0,\ \mathbf{u}\in \mathbf{D}, \text{ such that }  u_i >0 \text{ for all } i=1,2,3 \right\rbrace.
\end{equation}
Define
\begin{align*}
\mathcal{N}=
\left \{\mathbf{u}=(u_1,u_2,u_3)\in \mathbf{D}: \mathbf{u}\neq \mathbf{0} \text{ and } I'(\mathbf{u})(\mathbf{u})=0 \right \},
\end{align*}
and
\begin{align*}
\mathcal{M}=
\left \{\mathbf{u}=(u_1,u_2,u_3)\in \mathbf{D}: u_i\neq 0 \text{ and } \partial_iI(\mathbf{u})u_i=0 \text{ for every } i=1,2,3\right \}.
\end{align*}
Let
\begin{equation}\label{1.12}
A:= \inf_{\mathbf{u}\in \mathcal{M}}I(\mathbf{u})=\inf _{(u_1,u_2,u_3)\in \mathcal{M}}\frac{1}{N}\sum_{i=1}^3\|u_i\|_i^2, \quad
~\mathcal{A}:= \inf_{\mathbf{u}\in \mathcal{N}}I(\mathbf{u}),
\end{equation}
it is standard to show that $0<\mathcal{A}\leq A$. Next, we introduce the main results of this paper.

\subsection{Main results}

Our first result is related to the existence of {\it ground state solutions} to \eqref{mainsystem}.
Before proceeding, we introduce the following condition about the coupled coefficient $\beta_{ij}$
\begin{equation}\label{Conditions}
\sum_{i,j=1}^{3}\int_{\Omega}\beta_{ij}|\varphi_i|^{\frac{2^*}{2}}|\varphi_j|^{\frac{2^*}{2}}>0,\ ~~ \forall (\varphi_{1},\varphi_{2}, \varphi_{3})\in \mathbf{D}\setminus \{\mathbf{0}\},
\end{equation}
where we assume that $\beta_{ii}=1, ~i=1,2,3$.
Note that the condition \eqref{Conditions} is true if $\beta_{ij}\geq 0$ for any $i\neq j$. In fact, the condition \eqref{Conditions} holds if the matrix $\mathcal{B}:=(\beta_{ij})_{1\leq i,j\leq d}$ is positively definite.
\begin{thm}\label{thm1}

\begin{itemize}
  \item[(1)] Assume that $N\geq 3$, $\lambda_i\in (0, \Lambda_N)$ and the condition \eqref{Conditions} holds, then $\mathcal{A}$ is attained by a nonnegative solution $\mathbf{u}$. Moreover, $\mathbf{u}$ is a {\it ground state solution} of \eqref{mainsystem}.

  \item[(2)] Assume that $N=4$, $\lambda_i\in (0, \Lambda_4)$ and
  \begin{equation*}
  0<\beta_{ij}< \frac{\sqrt{2}}{2}, \quad \text{ for any } i\neq j,
  \end{equation*}
  then all ground state solutions of system \eqref{mainsystem} are semi-trivial, i.e., ${\mathcal{A}}<A$.

  \item[(3)] Assume that $N=3$, $\lambda_i\in (0, \Lambda_3)$ and
  \begin{equation*}
  0< \beta_{ij} < 1, \quad \text{ for any } i\neq j,
  \end{equation*}
  then all ground state solutions of system \eqref{mainsystem} are semi-trivial, i.e., ${\mathcal{A}}<A$.

  \item[(4)] Assume that $N\geq 5$, $\lambda_i\in (0, \Lambda_N)$ and
  \begin{equation}
  \beta_{ij}\geq 0, ~\text{ for any } i\neq j,
  \end{equation}
  then \eqref{mainsystem} has a positive ground state solution $\mathbf{u}$. Moreover, $\mathbf{u}$ is also a least energy positive solution in this case, and $\mathcal{A}=A$.
  \end{itemize}
\end{thm}

\begin{remark}
\begin{itemize}
  \item[(1)]  In \cite{Guo-Luo-Zou2021}, a similar result to Theorem \ref{thm1}-(1) was obtained only for the purely cooperative case ($\beta_{ij}\geq 0$, for any $i\neq j$). Here, Theorem \ref{thm1}-(1) holds under the condition \eqref{Conditions}, which includes not only the purely cooperative case but also the case of simultaneous cooperation and competition.

  \item[(2)] When $N=3,4$, Theorem \ref{thm1}-(2) and (3) show that all ground state solutions of system \eqref{mainsystem} are semi-trivial for $\beta_{ij}\geq 0$ small. However, when $N\geq 5$, Theorem \ref{thm1}-(4) states that system \eqref{mainsystem} has a positive ground state solution for any $\beta_{ij}\geq 0$ ($i\neq j$). Therefore, the structure of ground state solution to system \eqref{mainsystem} differs significantly between $N=3, 4$ and higher dimensions $N\geq 5$.
  \item[(3)] Although we only consider the case $d=3$, Theorem \ref{thm1} is true for any number of equations $d\geq 2$
\end{itemize}
\end{remark}

The proof of Theorem \ref{thm1}-(2) and (3) is inspired form \cite[Theorem 1.7]{Correia-Oliveria-Tavares2016}; however, the assumption $\beta_{ij}=b$ required therein is not needed in our proof. We highlight that there are technical differences between $N=3$ and $N=4$, please see the proof of Theorem \ref{thm1}-(2) and (3) for more details.

\bigbreak
The above theorem reveals that all ground state solutions of system \eqref{mainsystem} are semi-trivial for weakly cooperative case ($\beta_{ij}>0$ small, for any $i\neq j$) when $N=3,4$. A natural question is whether system \eqref{mainsystem} admits higher energy positive solutions for the weakly cooperative case when $N=3,4$ and whether the level $A$ is attained, where the level $A$ is defined in \eqref{1.12}. Here, we have the following results.

\begin{thm}\label{thm2}
\begin{itemize}
  \item[(1)] Assume that $N=4$, $\lambda_i\in (0, \Lambda_4)$ and there exists $\widetilde{\beta}>0$ such that
  \begin{equation}
  0\leq \beta_{ij}< \widetilde{\beta}, ~\text{ for any } i\neq j,
  \end{equation}
  then $A=I(\mathbf{u})$ and system \eqref{mainsystem} has a least energy positive solution $\mathbf{u}$.

  \item[(2)] Assume that $N=3$, $\lambda_i\in (0, \Lambda_3)$ and there exists $\widehat{\beta}>0$ such that
  \begin{equation}
  0\leq \beta_{ij}< \widehat{\beta}, ~\text{ for any } i\neq j,
  \end{equation}
  then $A=I(\mathbf{u})$ and system \eqref{mainsystem} has a least energy positive solution.
\end{itemize}
\end{thm}

\begin{remark}
The constant $\widetilde{\beta}$ only depends on $\mathcal{S}, \lambda_{i}, i=1,2, 3$, and the explicit expression for $\widetilde{\beta}$ can be found in \eqref{tem529-12}; $\widehat{\beta}$ only depends on $\mathcal{S}, \lambda_{i}, i=1,2, 3$, and the explicit expression can be found in \eqref{tem529-4}.
\end{remark}

\begin{remark}
For the two components case, Chen and Zou \cite{Chen-Zou2015} studied the following general system
\begin{equation}\label{twoequations}
\begin{cases}
-\Delta u -\frac{\lambda_{1}}{|x|^2}u=|u|^{2^*-2}u+
\nu \alpha u^{\alpha-1}v^{\beta}, x \in \mathbb{R}^N,\\
-\Delta v -\frac{\lambda_{2}}{|x|^2}v=|v|^{2^*-2}u+\nu \beta u^{\alpha}v^{\beta-1}, ~x\in \mathbb{R}^N,\\
u, v \in D^{1,2}(\mathbb{R}^N), u, v >0 \text{ in } \mathbb{R}^N \backslash \{0\},
\end{cases}
\end{equation}
where $N\geq 3$, $\alpha+\beta=2^*$, $2^*=2N/(N-2)$. In that paper, when $N=3$ the authors proved that there exists a positive constant $\beta_1$ such that \eqref{twoequations} has a least energy positive solution if $\beta\in (0,\beta_1)$. The accurate definition of $\beta_1$ is not given in \cite{Chen-Zou2015}. Here, we show that there exists $\widehat{\beta}>0$ such that \eqref{twoequations} with $\alpha=\beta=3$ has a least energy positive solution if $\nu\in (0,\widehat{\beta})$ (see Corollary \ref{1107-1}) and give a precise definition of $\widehat{\beta}$ by a different method.
\end{remark}

The proof of Theorem \ref{thm2} (1) is inspired by the one of \cite[Theorem 1.2]{Chen-Zou2015}. However, the method in \cite{Chen-Zou2015} cannot be used here directly due to the presence of multi-components, and we need some crucial modifications for our proof. Firstly, we need establish new estimates (see Lemma \ref{1021-7} and Lemma \ref{Energyestimates4}). Secondly, unlike \cite{Chen-Zou2015}, our method does
not require a comparison between the level $A$ and the ground state level of the limiting system ($\lambda_i=0$). Thirdly, we mention that dealing with many components does not allow to perform some explicit computations as in the two equation case. For instance, when projecting in
the Nehari manifold, one can not, in general, obtain the explicit expression of the coefficients; so, we depend on the notion of strictly diagonally dominant matrices as
in \cite{Soave-Tavares2016,TavaresYou2019,Hugo-You-Zou2022} to show positive definiteness
of the coupled matrices, see Lemma \ref{estimation3}.

We employ the idea in proving Theorem \ref{thm2}-(1) to finish the proof of Theorem \ref{thm2}-(2). However, there are technical differences between $N=3$ and $N=4$. For example, when projecting on the Nehari manifold, the equations the coefficients comply to are a linear system for $N=4$ (see \eqref{tem1125-6} and \eqref{tem0511-5}) and the Cramer's rule can be used. While the corresponding equations is a nonlinear system for $N=3$ (see \eqref{2325-9}) and the Cramer's rule is not available. Thus, we introduced a new technique to deal with the case $N=3$ (see the proof of Lemma \ref{Energyestimates6} and Lemma \ref{estimation4-4}).

\medbreak
Theorem \ref{thm2} shows that system \eqref{mainsystem} has a least energy positive solution for weakly cooperative case ($\beta_{ij}>0$ small, for any $i\neq j$) when $N=3,4$. While \cite[Theorem 1.1]{Guo-Luo-Zou2021} states that system \eqref{mainsystem} with $N\geq 3$ has no least energy positive solutions for the purely competitive case ($\beta_{ij}<0$ for any $i\neq j$). A nature question is whether system \eqref{mainsystem} has a least energy positive solution for the mixed case. For the case $\beta_{12}>0$, $\beta_{13}, \beta_{23}<0$, we have the following result.

\begin{thm}\label{thm3}
\begin{itemize}
  \item[(1)] Assume that $N=4$, $\lambda_i\in (0, \Lambda_4)$, $\lambda_1=\lambda_2$ and
  $$
  0<\beta_{12}< \widetilde{\beta}, \quad \beta_{13}<0, \quad \beta_{23}<0,
  $$
  then $A$ is not attained, which implies that \eqref{mainsystem} has no least energy positive solutions.

  \item[(2)] Assume that $N=3$, $\lambda_i\in (0, \Lambda_3)$, $\lambda_1=\lambda_2$ and
  $$
  0<\beta_{12}< \widehat{\beta}, \quad \beta_{13}<0, \quad \beta_{23}<0,
  $$
  then $A$ is not achieved, which means that \eqref{mainsystem} has no least energy positive solutions.
\end{itemize}
\end{thm}

\begin{remark}
We mention that the condition $\lambda_1=\lambda_2$ is only used in Lemma \ref{tem1010-1} and Lemma \ref{tem1206-1}. We conjecture that the condition $\lambda_1=\lambda_2$ in Theorem \ref{thm3} can be removed.
\end{remark}

The proof of Theorem \ref{thm3}-(1) and (2) is inspired by \cite[Theorem 3]{Lin-Wei2005}, where the authors proved the nonexistence of least energy positive solutions to Schr\"{o}dinger systems for the subcritical case in $\mathbb{R}^N$. But in \cite{Lin-Wei2005} the matrix $\mathcal{B}:=(\beta_{ij})_{1\leq i,j\leq d}$ is assumed to be positively definite. Here, the positive definiteness of the matrix $\mathcal{B}$ can be removed. Thus, the method in \cite{Lin-Wei2005} cannot be used here and we need some additional new tricks to deal with this problem. We mainly depend on the notion of strictly diagonally dominant matrices to show positive definiteness of the coupled matrices and control the growth of the energy functional $I$. Compared with the proof of Theorem \ref{thm3}-(1), the proof of Theorem \ref{thm3}-(2) is more complicated and we use different techniques to deal with the case $N=3$ (see Lemma \ref{tem1012-12} and Lemma \ref{tem1012-20}).

\bigbreak
When $N\geq5$, the Nehari Manifold $\mathcal{M}$ is not a closed set and so is not a natural constraint. Then in turn we consider a sub-manifold of $\mathcal{M}$
\begin{align*}
\mathcal{N}_{12,3}:&=\Bigg\{(u_1,u_2)\neq \mathbf{0}, u_3\neq 0: \|u_3\|_3^2=\sum_{i=1}^3\int_{\mathbb{R}^N}\beta_{i3}|u_i|^{\frac{2^*}{2}}|u_3|^{\frac{2^*}{2}}\\
&\,\,\,\,\,\,\, \,\,\,\,\,\,\, \sum_{i=1}^2\|u_i\|_i^2=\sum_{i,j=1}^2\int_{\mathbb{R}^N}\beta_{ij}|u_i|^{\frac{2^*}{2}}|u_j|^{\frac{2^*}{2}}
+\sum_{i=1}^2\int_{\mathbb{R}^N}\beta_{i3}|u_i|^{\frac{2^*}{2}}|u_3|^{\frac{2^*}{2}} \Bigg \}
\end{align*}
and
\begin{equation}\label{tem512-5}
\widetilde{A}:=\inf_{\mathbf{u}\in \mathcal{N}_{12,3}}I(u_1,u_2,u_3).
\end{equation}
Parallel with Theorem \ref{thm3}, for the case $N\geq 5$ we have
\begin{thm}\label{thm4}
Assume that $N\geq 5$ and $\lambda_i\in (0, \Lambda_N)$ and
$$
\beta_{12}>0, \quad \beta_{13}<0, \quad \beta_{23}<0,
$$
then the level $\widetilde{A}$ is not attained and \eqref{mainsystem} has no least energy solutions $\mathbf{u}=(u_1,u_2,u_3)$ with $(u_1,u_2)\neq \mathbf{0}, u_3\neq 0$.
\end{thm}


\begin{remark}
At the end of this section, we would like to give the following remark. In the papers \cite{LiuYouZou2023,TavaresYou2019,Hugo-You-Zou2022}, the authors consider the following critical Schr\"{o}dinger system
\begin{equation}\label{system4}
\begin{cases}
-\Delta u_{i}-\lambda_iu_{i}=|u_i|^{2^*-2}u_i+\sum_{j\neq i}^{d}\beta_{ij}|u_{j}|^{\frac{2^*}{2}}|u_i|^{\frac{2^*}{2}-2}u_i \text{ in } \Omega,\\
u_i\geq 0 \text{ in } \Omega, u_i=0 \text{ on } \partial \Omega, i=1,\cdots,d,
\end{cases}
\end{equation}
where $\Omega$ is a bounded domain in $\mathbb{R}^N$, $2^*=\frac{2N}{N-2}$, $N\geq3$, $\beta_{ii}>0, ~i=1,\cdots,d$, $\beta_{ij}=\beta_{ji}$ for any $i \neq j$ (the cases $N=3$, $N=4$ and $N\ge5$ were considered respectively in \cite{LiuYouZou2023}, \cite{TavaresYou2019} and \cite{Hugo-You-Zou2022}). However, problem \eqref{mainsystem} is considerably different from problem \eqref{system4}.
\begin{itemize}
\item [(i)] On the one hand, as implicitly pointed out by Chen and Zou in \cite{Chen-Zou2015}, the Palais-Smale condition of problem \eqref{mainsystem} does not hold for any energy level, so that it makes problem \eqref{mainsystem} very complicated. In \cite{LiuYouZou2023,TavaresYou2019,Hugo-You-Zou2022}, it shows that the Palais-Smale condition for problem \eqref{system4} holds for some special ranges of the energy level thanks to some accurate energy estimates.
\item [(ii)] On the other hand, the structure of least energy solutions to system \eqref{mainsystem} is much different from that of problem \eqref{system4}. For example, when $N=4$, least energy positive solutions of system \eqref{system4} are obtained in \cite{TavaresYou2019} if
\begin{equation}\label{en}
\beta_{13}<0, \beta_{23}<0, \beta_{12}>0\,\,\mbox{small},
\end{equation}
but in this paper, we show that system \eqref{mainsystem} with $N=4$ has no least energy positive solutions under the assumption \eqref{en}(see Theorem \ref{thm3}-(1)). When $N\geq 5$, it is showed in  \cite[Theorem 1.1]{Hugo-You-Zou2022}($d=3, m=2$) that the level $\mathcal{A}$ is attained if
\begin{equation}\label{en1}
\beta_{13}<0, \beta_{23}<0, \beta_{12}>0,
\end{equation}
whereas, in this paper, we show that the level $\mathcal{A}$ is not attained under the condition \eqref{en1} (see Theorem \ref{thm4}). The reason underlying is the fact that system \eqref{mainsystem} is invariant in $\R^N$ under the transformation
$$
(u_1, u_2, u_3)\mapsto (a^{\frac{N-2}{2}}u_1(a\cdot),a^{\frac{N-2}{2}}u_2(a\cdot),a^{\frac{N-2}{2}}u_3(a\cdot)),\,\,a>0.
$$
\end{itemize}
\end{remark}

\subsection{Further notations}

\begin{itemize}

  \item Denote by $|\cdot|_p$ the norm of the $L^p(\mathbb{R}^N)$, $1\leq p \leq \infty$.

  \item  $\mathbf{D}= D^{1,2}(\mathbb{R}^N; \R^3)$, with norm $\|\mathbf{u}\|^{2}=\sum_{i=1}^{3}\|u_{i}\|_{i}^{2}$, where
  \begin{equation*}
  \|u\|_i^2:=\int_{\mathbb{R}^N}\left(|\nabla u|^2-\frac{\lambda_{i}}{|x|^2} u^{2}\right).
  \end{equation*}

  \item Set
  \begin{equation}\label{Constant-1125}
  S:=\min\{S_1, S_2, S_3\},
  \end{equation}
  where $S_i$ is defined in \eqref{tem529-8-4}.

  \item We always assume that $\beta_{ii}=1$, for any $i=1,2,3.$

  \item Given $\Gamma\subseteq\{1,2,3\}$ with $|\Gamma|=2,3$, consider the following subsystem
  \begin{equation}\label{subsystem}
  \begin{cases}
  -\Delta u_{i}-\frac{\lambda_{i}}{|x|^2}u_{i}=|u_i|^{2^*-2}u_i+\sum_{j\neq i}^{3}\beta_{ij}|u_{j}|^
  {\frac{2^*}{2}}|u_i|^{\frac{2^*}{2}-2}u_i, ~x\in \mathbb{R}^N, ~i\in \Gamma,\\
  u_i\in D^{1,2}(\mathbb{R}^N), i\in \Gamma,
  \end{cases}
  \end{equation}
  and define
  \begin{equation}\label{1.10-4}
  I_\Gamma(\mathbf{u}_\Gamma):=\frac{1}{2}\sum_{i\in \Gamma} \left\|u_i \right\|_i^2-\frac{1}{2^*}\sum_{i,j\in \Gamma} \int_{\mathbb{R}^N}\beta_{ij} |u_i|^
  {\frac{2^*}{2}}|u_j|^{\frac{2^*}{2}},
  \end{equation}
  \begin{equation}\label{tem107-4}
  \mathcal{M}_{\Gamma}=
  \left \{\mathbf{u}_\Gamma\in D^{1,2}(\mathbb{R}^N; \R^{|\Gamma|}): u_i\neq 0,  \|u_{i}\|_i^2=\sum_{j\in \Gamma} \int_{\mathbb{R}^N}\beta_{ij}|u_{i}|^{\frac{2^*}{2}}|u_{j}|^{\frac{2^*}{2}}, \ i\in \Gamma\right \}.
  \end{equation}
  Denote
  \begin{equation}\label{tem529-10}
  A_{\Gamma}:= \inf_{\mathbf{u}\in \mathcal{M}_{\Gamma}}I_\Gamma(\mathbf{u}_\Gamma)=\inf _{\mathbf{u}_\Gamma\in \mathcal{M}_{\Gamma}}\frac{1}{N}\sum_{i\in \Gamma}\|u_i\|_i^2,
  \end{equation}
  and $A_{\Gamma}=A, \mathcal{M}_{\Gamma}=\mathcal{M}$ when $\Gamma=\{1,2,3\}$.

  \item Given $\Gamma\subseteq\{1,2,3\}$ with $|\Gamma|=2,3$, $\mathbf{u}=(u_1,\cdots,u_{|\Gamma|})\in D^{1,2}(\mathbb{R}^N; \R^{|\Gamma|})$, define the $|\Gamma|\times |\Gamma|$ matrix
    \begin{equation}\label{Matrix-Def-2}
    M_{B}[u_1,\cdots,u_{|\Gamma|}]:=\left( M_{B}[u_1,\cdots,u_{|\Gamma|}]_{hk}\right)_{ h,k\in \Gamma}:= \left( \int_{\mathbb{R}^N}\beta_{hk}|u_{h}|^{\frac{2^*}{2}}|u_{k}|^{\frac{2^*}{2}}\right)_{h,k\in \Gamma}.
    \end{equation}

\end{itemize}

\section{Ground state solutions}
\renewcommand{\theequation}{2.\arabic{equation}}
\subsection{Existence of ground state solutions}
Denote
\begin{equation*}
E(u_1,u_2,u_3):=\sum_{i=1}^3\left(|\nabla u_i|^{2}-\frac{\lambda_{i}}{|x|^2} u_i^{2}\right), \quad
F(u_1,u_2,u_3):=\sum_{i,j=1}^3 \beta_{ij}|u_{i}|^{\frac{2^*}{2}}|u_{j}|^{\frac{2^*}{2}}.
\end{equation*}
For any $\mathbf{u}\in \mathbf{D}\setminus \{\mathbf{0}\}$, by the condition \eqref{Conditions}, we know that
\begin{equation*}
\int_{\mathbb{R}^N}F(u_1,u_2,u_3)>0.
\end{equation*}
Then, thanks to the definition of $\mathcal{A}$ in \eqref{1.12},
\begin{equation*}
\mathcal{A}=\inf_{\mathbf{u}\in \mathbf{D}\setminus \{\mathbf{0}\}}\frac{1}{N}
\left[\frac{\int_{\mathbb{R}^N}E(u_1,u_2,u_3)}{\int_{\mathbb{R}^N}F(u_1,u_2,u_3)}\right]^{\frac{N}{2}}
\end{equation*}
and then
\begin{equation}\label{tem105-1}
\int_{\mathbb{R}^N}E(u_1,u_2,u_3)\geq (N\mathcal{A})^{\frac{2}{N}}\left(\int_{\mathbb{R}^N}F(u_1,u_2,u_3)\right)^{\frac{2}{2^*}}, ~~\forall ~\mathbf{u}\in \mathbf{D}.
\end{equation}
\begin{proof}[\bf The proof of Theorem \ref{thm1}-(1)]
Based on the inequality \eqref{tem105-1}, following the proof of \cite[Lemma 3.3]{Guo-Luo-Zou2021} we can get that $\mathcal{A}$ is attained by a nonnegative solution $\mathbf{u}$, which is a ground state solution.

\end{proof}

\subsection{Nonexistence of fully nontrivial ground state solutions for $N=3,4$}
In this subsection, we prove (2) and (3) of Theorem \ref{thm1}.
\begin{proof}[\bf The proof of Theorem \ref{thm1}-(2)]
The following proof is inspired by \cite[Theorem 1.7]{Correia-Oliveria-Tavares2016}, and we sketch it here for completeness. Based on Theorem \ref{thm1}-(1), let $\mathbf{u}=(u_1,u_2,u_3)\in \mathcal{N}$ be a ground state solution of system \eqref{mainsystem}. Assume now that the ground state solution $\mathbf{u}$ is fully nontrivial, that is $u_i\neq 0$, ~$i=1,2,3$.
Note that
\begin{equation}
(tu_1,0,0) \in \mathcal{N}\  \Leftrightarrow \  t^2=\cfrac{\left\| u_1\right\|_1^2 }{|u_1|_4^4},
\end{equation}
Since $\mathbf{u}=(u_1,u_2,u_3)$ is a ground state solution of system \eqref{mainsystem}, then $I(u_1,u_2,u_3)\leq I(tu_1,0,0)$, which yields that
\begin{equation*}
\sum_{i=1}^3|u_i|_4^4+2\sum_{i<j}\beta_{ij}|u_iu_j|_2^2\leq t^4|u_1|_4^4.
\end{equation*}
Then
\begin{equation*}
|u_1|_4^4\left(\sum_{i=1}^3|u_i|_4^4+2\sum_{i<j}\beta_{ij}|u_iu_j|_2^2\right)\leq
\left(|u_1|_4^4+\sum_{i=2}^3\beta_{1j}|u_1u_j|_2^2\right)^2,
\end{equation*}
which implies that
\begin{equation*}
\sum_{i=2}^3|u_i|_4^4 \leq \left(\sum_{i=2}^3\beta_{1j}|u_j|_4^2\right)^2\leq 2\left(\max_{i\neq j}\{\beta_{ij}\}\right)^2 \sum_{i=2}^3|u_i|_4^4,
\end{equation*}
and so
\begin{equation*}
\max_{i\neq j}\{\beta_{ij}\}\geq \frac{\sqrt{2}}{2}.
\end{equation*}
Therefore, if
$$
0< \beta_{ij} < \frac{\sqrt{2}}{2},
$$
all ground state solutions of system \eqref{mainsystem} are semi-trivial.
\end{proof}

\begin{proof}[\bf Conclusion of the proof of Theorem \ref{thm1}-(3)]
Based on Theorem \ref{thm1}-(1), let $\mathbf{u}=(u_1,u_2,u_3)\in \mathcal{N}$ be a ground state solution of system \eqref{mainsystem}. Assume now that the ground state solution $\mathbf{u}$ is fully nontrivial, that is $u_i\neq 0$,$i=1,2,3$.
Note that
\begin{equation}
(tu_1,0,0) \in \mathcal{N}\  \Leftrightarrow \  t^4=\cfrac{\left\| u_1\right\|_1^2 }{|u_1|_6^6},
\end{equation}
Since $\mathbf{u}=(u_1,u_2,u_3)$ is a ground state solution of system \eqref{mainsystem}, then $I(u_1,u_2,u_3)\leq I(tu_1,0,0)$, that is
\begin{equation*}
\sum_{i=1}^3|u_i|_6^6+2\sum_{i<j}\beta_{ij}|u_iu_j|_3^3\leq t^2\|u_1\|_1^2.
\end{equation*}
Therefore,
\begin{equation*}
|u_1|_6^6\left(\sum_{i=1}^3|u_i|_6^6+2\sum_{i<j}\beta_{ij}|u_iu_j|_3^3\right)^2\leq
\left(|u_1|_6^6+ \sum_{j=2}^3\beta_{1j}|u_1u_j|_3^3\right)^3,
\end{equation*}
which implies that
\begin{align*}
 |u_1|_6^6\Big( \sum_{i=1}^3|u_i|_6^6 &+ 2\sum_{j=2}^3\beta_{1j}|u_1u_j|_3^3  +2\beta_{23}|u_2u_3|_3^3 \Big)^2  \leq |u_1|_6^6+(\sum_{j=2}^3\beta_{1j}|u_1u_j|_3^3)^3\\
   & + 3|u_1|_6^{12}\sum_{j=2}^3\beta_{1j}|u_1u_j|_3^3
  +3|u_1|_6^{6}\left(\sum_{j=2}^3\beta_{1j}|u_1u_j|_3^3\right)^2,
\end{align*}
that is,
\begin{align}\label{tem105-9}
|u_1|_6^6 \Big( \sum_{i=2}^3|u_i|_6^6\Big)^2+2|u_1|_6^{12}\sum_{i=2}^3|u_i|_6^6
& \leq \left(\sum_{j=2}^3\beta_{1j}|u_1u_j|_3^3\right)^3\leq \left(\sum_{j=2}^3\beta_{1j}|u_1|_6^3|u_j|_6^3\right)^3\nonumber\\
& \leq |u_1|_6^9 \left[\left(\sum_{j=2}^3\beta_{1j}|u_j|_6^3\right)^2\right]^{\frac{3}{2}} \\
& \leq 2^{\frac{3}{2}}|u_1|_6^9 \left(\max_{i\neq j}\{\beta_{ij}\}\right)^3\Big( \sum_{i=2}^3|u_i|_6^6\Big)^{\frac{3}{2}}.\nonumber
\end{align}
On the other hand,
\begin{equation*}
|u_1|_6^6 \Big( \sum_{i=2}^3|u_i|_6^6\Big)^2+2|u_1|_6^{12}\sum_{i=2}^3|u_i|_6^6\geq
2^{\frac{3}{2}}|u_1|_6^9\left( \sum_{i=2}^3|u_i|_6^6\right)^{\frac{3}{2}},
\end{equation*}
Thus, combining this with \eqref{tem105-9} we see that
\begin{equation*}
\max_{i\neq j}\{\beta_{ij}\}\geq 1.
\end{equation*}
Therefore, if
$$
0< \beta_{ij} < 1,
$$
all ground state solutions of system \eqref{mainsystem} are semi-trivial.
\end{proof}


\begin{proof}[\bf The proof of Theorem \ref{thm1}-(4)]
Note that $\|u\|_i^2=\int_{\mathbb{R}^N}\left(|\nabla u|^2-\frac{\lambda_{i}}{|x|^2} u^{2}\right)$ is a norm and the exponent $2^*/2\in (1,2)$ when $N\geq 5$, then the proof is same as that of \cite[Theorem 1.2]{TO2016}, and so we omit it.
\end{proof}

\section{Existence of least energy positive solutions}
\renewcommand{\theequation}{3.\arabic{equation}}
\subsection{The case for $N=4$}

\subsubsection{Preliminary results}
\begin{lemma}\label{EnergyEstimate1}
Set
\begin{equation}\label{Constant-2-1}
\overline{C}=\frac{3}{N} \mathcal{S}^{\frac{N}{2}},
\end{equation}
then we have
\begin{equation*}
A \leq \overline{C}.
\end{equation*}
\end{lemma}
\begin{proof}[\bf Proof]
Taking $v_1,v_2,v_3 \not\equiv 0$ such that $v_i \cdot v_j \equiv 0$ when $i\neq j$ and setting $\widehat{v}_i=t_i v_i$, where
\begin{equation}\label{Bounded-3}
t_i^{2^*-2}=\frac{\|v_i\|_i^2}{|v_i|_{2^*}^{2^*}}.
\end{equation}
we have $(\widehat{v}_1,\widehat{v}_2,\widehat{v}_3)\in \mathcal{M}$. By \eqref{Bounded-3},
\begin{align}\label{Bounded-3-1}
 A & \leq I(\widehat{v}_1,\widehat{v}_2,\widehat{v}_3)=\frac{1}{N}\sum_{i=1}^3 \|\widehat{v}_i\|_i^2= \frac{1}{N}\sum_{i=1}^3 t_i^2\|v_i\|_i^2\nonumber\\
   & = \frac{1}{N}\sum_{i=1}^3\left(\frac{\|v_i\|_i^2}{|v_i|_{2^*}^{2^*}}\right)^{\frac{2}{2^*-2}}\|v_i\|_i^2
   \leq\frac{1}{N}\sum_{i=1}^3\left(\frac{\|v_i\|^2}{|v_i|_{2^*}^{2}}\right)^{\frac{N}{2}}.
\end{align}
Denote
\begin{equation*}
\widehat{S}(\Omega):=\inf_{u\in H_{0}^{1}(\Omega)\setminus \{0\}}\frac{\|u\|^2}{|u|_{2^*}^{2}}.
\end{equation*}
By \cite{Terracini1996} we know that $\widehat{S}(\Omega)=\mathcal{S}$. Combining this with \eqref{Bounded-3-1},
\begin{align}\label{Bounded-3-2}
 A & \leq \frac{1}{N}\sum_{i=1}^3\inf_{\Omega \supset\Omega_1,\Omega_2,\Omega_3\neq \emptyset \atop \Omega_{i}\cap\Omega_{j}=\emptyset, i\neq j}\widehat{S}^{\frac{N}{2}}(\Omega_i)=\frac{3}{N} \mathcal{S}^{\frac{N}{2}}=\overline{C},
\end{align}
where $\overline{C}$ is defined in \eqref{Constant-2-1}.
\end{proof}

Denote
\begin{equation}\label{tem529-13}
\beta_1:=\frac{S^2}{16\overline{C}}.
\end{equation}

\begin{lemma}\label{Estimate2}
Assume that $N=4$ and $0\leq\beta_{ij}< \beta_1$, for any $i\neq j$, then there exist $C_1,C_2>0$ dependent on $\lambda_i, S, \overline{C}$, such that for any $\mathbf{u} \in \mathcal{M}$ with $I(\mathbf{u})\leq 2\overline{C}$,
\begin{equation*}
C_1\leq|u_i|^2_{4}\leq C_2, ~i=1,2,3.
\end{equation*}
\end{lemma}
\begin{proof}[\bf Proof]
Note that $\mathbf{u} \in \mathcal{M}$ and $I(\mathbf{u})\leq 2\overline{C}$, then we have
\begin{equation*}
S|u_i|_4^2\leq \|u_i\|_i^2\leq \sum_{i=1}^3\|u_i\|_i^2\leq 8\overline{C},
\end{equation*}
which implies that
\begin{equation*}
|u_i|_4^2\leq C_1:=8\overline{C}/S.
\end{equation*}
On the other hand, since
\begin{equation*}
\|u_i\|_i^2 = \sum_{j=1}^3 \int_{\mathbb{R}^4}\beta_{ij}|u_i|^2|u_j|^2,
\end{equation*}
we have
\begin{align*}
  S|u_i|_4^2 & \leq \sum_{j=1}^3 \int_{\mathbb{R}^4}\beta_{ij}|u_i|^2|u_j|^2
   \leq |u_i|_4^4 + \sum_{j\neq i}\beta_{ij}|u_i|_4^2|u_j|_4^2\\
   & \leq |u_i|_4^4 +\beta_1\sum_{j\neq i}S^{-1}\|u_j\|_j^2|u_i|_4^2\\
   & \leq |u_i|_4^4+\beta_1 S^{-1}8\overline{C}|u_i|_4^2=|u_i|_4^4+\frac{S}{2}|u_i|_4^2.
\end{align*}
Thus, we know that
\begin{equation*}
|u_i|_4^2\geq C_2:=\frac{S}{2}.
\end{equation*}
\end{proof}

Denote
\begin{equation}\label{tem529-14}
\beta_2:=\min\left\{\beta_1, \frac{SC_1^2}{16\overline{C}C_2}\right\}.
\end{equation}

\begin{lemma}\label{PS-sequence}
Suppose that $N=4$ and $0\leq\beta_{ij}< \beta_2$, for any $i\neq j$. Let $\widehat{\mathbf{u}}_n=(\widehat{u}_1^n,\widehat{u}_2^n,\widehat{u}_3^n)
\in\mathcal{M}$ be a minimizing sequence of $A$, then there exists a sequence $\{\mathbf{u}_n\}\in\mathcal{M}$ such that
\begin{equation*}
\lim_{n\rightarrow \infty} I(\mathbf{u}_n)=A, \quad \lim_{n\rightarrow \infty} I'(\mathbf{u}_n)=0.
\end{equation*}
\end{lemma}
\begin{proof}[\bf Proof]
We assume that $I(\mathbf{u}_n)\leq 2\overline{C}$ for $n$ large enough. Next, we claim that the matrix $M_B[\widehat{u}_1^n,\widehat{u}_2^n,\widehat{u}_3^n]$ is positively definite, where $M_B[\widehat{u}_1^n,\widehat{u}_2^n,\widehat{u}_3^n]$ is defined in \eqref{Matrix-Def-2}. In fact, for any $i=1,2,3$
\begin{align*}
  &\int_{\mathbb{R}^4}\beta_{ii}|\widehat{u}_i^n|^4-\sum_{j\neq i}\left|\int_{\mathbb{R}^4}\beta_{ij}|\widehat{u}_i^n|^2|\widehat{u}_j^n|^2\right|\\
   & \geq C_1^2- \beta_2\sum_{j\neq i}|\widehat{u}_i^n|_4^2|\widehat{u}_j^n|_4^2\geq C_1^2-\beta_2S \sum_{j\neq i}|\widehat{u}_i^n|_4^2\|\widehat{u}_j^n\|_j^2\\
   & \geq C_1^2- \beta_28\overline{C} S^{-1}|\widehat{u}_i^n|_4^2\geq C_1^2-\beta_28\overline{C}S^{-1}C_2\geq \frac{1}{2}C_1^2,
\end{align*}
which yields that the matrix $M_B[\widehat{u}_1^n,\widehat{u}_2^n,\widehat{u}_3^n]$ is strictly diagonally dominant. Thus, by the Gershgorin circle theorem we know that the matrix $M_B[\widehat{u}_1^n,\widehat{u}_2^n,\widehat{u}_3^n]$ is positive definite. Based on these arguments, following the proof of \cite[Proposition 3.13]{TavaresYou2019}, there exists a sequence $\{\mathbf{u}_n\}\in\mathcal{M}$ such that
\begin{equation*}
\lim_{n\rightarrow \infty} I(\mathbf{u}_n)=A, \quad \lim_{n\rightarrow \infty} I'(\mathbf{u}_n)=0.
\end{equation*}

\end{proof}

\subsubsection{Energy estimates}
In this subsection, we establish energy estimates.
Denote
\begin{equation}\label{tem529-12}
\beta'_3:= \frac{1}{2}\min_{i\neq j}\left\{\frac{1-\lambda_i}{1-\lambda_j}, \frac{(1-\lambda_i)^{\frac{3}{4}}(1-\lambda_j)^{\frac{3}{4}}}
{(1-\lambda_i)^{\frac{3}{2}}+(1-\lambda_j)^{\frac{3}{2}}}\right\}.
\end{equation}
 It follows from \cite[Lemma 3.1]{Chen-Zou2015} that
\begin{equation}\label{1019-8}
A_{ij}<A_i+A_j \text{ when } \beta_{ij}\in (0, \beta'_3), ~i\neq j.
\end{equation}
Consider
\begin{equation*}
f_{ij}(t_i,t_j):=\frac{1}{2}\sum_{i=1}^2t_i\|\omega^{ij}_i\|_i^2-\frac{1}{4}\sum_{i=1}^2t_i^{2}
\int_{\mathbb{R}^4}|\omega^{ij}_i|^{4}-\frac{1}{2}t_it_j\int_{\mathbb{R}^4}\beta_{ij}|\omega^{ij}_i|^2|\omega^{ij}_j|^2.
\end{equation*}
Set
\begin{equation*}
\beta_3:=\min\left\{1, \beta'_3,\min_{i\neq j}\left\{ \frac{S^4}{16(A_i+A_j)} \right\}  \right\}.
\end{equation*}
\begin{lemma}\label{1021-7}
Assume that $N=4$ and $\beta_{ij}\in (0, \beta_3),i\neq j$. Then we have
\begin{equation}\label{tab1122-1}
\max_{t_{i},t_{j}\geq0}f_{ij}(t_{i},t_{j})=f_{ij}(1,1)=A_{ij}, ~\text{ for any } i\neq j,
\end{equation}
and
\begin{equation}\label{1021-12}
A_{ij}>\max\{A_i, A_j\}, ~\text{ for any } i\neq j.
\end{equation}
where $A_{ij}$ is defined in \eqref{tem529-10} and $A_i$ is defined in \eqref{tem529-1}.
\end{lemma}
\begin{proof}[\bf Proof]
Without loss of generality, we will present that
$
\max_{t_{1},t_{2}\geq 0}f_{12}(t_{1},t_{2})=f_{12}(1,1)=A_{12}
$
and $A_{12}>\max\{A_1, A_2\}$.
Assume that $I_{12}(\omega^{12}_1,\omega^{12}_2)=A_{12}$.
By \eqref{1019-8} we know that $A_{12}=\frac{1}{4}(\|\omega^{12}_1\|_1^2+\|\omega^{12}_2\|_2^2)\leq A_1+A_2$, combining this with $\beta_{ij}\in (0, \beta_3)$, and following the proof of Lemma \ref{Estimate2} we get that
\begin{equation}\label{1019-10}
|\omega^{12}_i|_4^2\geq \frac{S}{2}, ~~i=1,2.
\end{equation}
Note that $\beta_{ij}>0$ for any $i\neq j$, then by \eqref{1019-10} we have
\begin{align}\label{1019-13}
 \max_{t_{1},t_{2}\geq0}f_{12}(t_{1},t_{2})  & \leq \frac{1}{2}t_1\|\omega^{12}_1\|_1^2-\frac{1}{4}t_1^{2}
\int_{\mathbb{R}^4}|\omega^{12}_1|^{4}+\frac{1}{2}t_2\|\omega^{12}_2\|_2^2-\frac{1}{4}t_2^{2}
\int_{\mathbb{R}^4}|\omega^{12}_2|^{4}\nonumber\\
   & \leq \frac{1}{4}\sum_{i=1}^2 \frac{\|\omega^{12}_i\|_i^2}{|\omega^{12}_i|_4^{4}}\|\omega^{12}_i\|_i^2
    \leq \frac{4(A_1+A_2)^2}{S^2}.
\end{align}
Thus, $f_{12}$ has a global maximum point $(\widetilde{t}_1, \widetilde{t}_2)$ in $\overline{\mathbb{R}_{+}^2}$. Assume that the global maximum point $(\widetilde{t}_1, \widetilde{t}_2)\in \partial \overline{\mathbb{R}_{+}^2}$, without loss of generality we assume that $\widetilde{t}_1=0$.
Thus, $\partial_2f_{12}(0, \widetilde{t}_2)=0$, that is
\begin{equation}\label{1019-12}
\|\omega^{12}_2\|_2^2=\widetilde{t}_2|\omega^{12}_2|_4^{4}.
\end{equation}
It follows from \eqref{1019-13} and \eqref{1019-12} that
\begin{equation}\label{1021-1}
\frac{1}{4}\widetilde{t}_2\|\omega^{12}_2\|_2^2=I_{12}\left(0,\sqrt{\widetilde{t}_2}\omega^{12}_2\right)\leq \max_{t_{1},t_{2}\geq0}f_{12}(t_{1},t_{2})\leq
\frac{4(A_1+A_2)^2}{S^2}.
\end{equation}
Note that $\beta_{ij}\in (0, \beta_3)$, then by \eqref{1021-1} we get that
\begin{align}\label{1021-2}
\widetilde{t}_2\beta_{12}|\omega^{12}_1\omega^{12}_2|_2^2 & \leq \beta_{12}\widetilde{t}_2\frac{\|\omega^{12}_1\|_1^2\|\omega^{12}_2\|_2^2}{S^2}\nonumber\\
& \leq4\beta_{12}\frac{(A_1+A_2)^2}{S^2}
\|\omega^{12}_1\|_1^2\leq \frac{1}{4}\|\omega^{12}_1\|_1^2.
\end{align}
It follows from \eqref{1021-2} that
\begin{align*}
  f_{12}(s, \widetilde{t}_2) - f_{12}(0, \widetilde{t}_2) & = \frac{1}{2}s\|\omega^{12}_1\|_1^2-\frac{1}{4}s^2|\omega^{12}_1|_4^4- \frac{1}{2}s\widetilde{t}_2\beta_{12}|\omega^{12}_1\omega^{12}_2|_2^2 \\
   & \geq \frac{1}{4}s\|\omega^{12}_1\|_1^2-\frac{1}{4}s^2|\omega^{12}_1|_4^4>0  ~\text{ for } s>0 \text{ small enough, }
\end{align*}
a contradiction. Therefore, the global maximum point $(\widetilde{t}_1, \widetilde{t}_2)$ belongs to $\mathbb{R}_{+}^2$.

On the other hand, we deduce from $\beta_{ij}<1$ that
\begin{equation*}
\int_{\mathbb{R}^4}|\omega^{12}_1|^4 \int_{\mathbb{R}^4}|\omega^{12}_2|^4-\left(\beta_{12}\int_{\mathbb{R}^4}|\omega^{12}_1|^2
|\omega^{12}_2|^2\right)^2>0,
\end{equation*}
which implies that the matrix $M_B[\omega_1,\omega_2]$ is positive definite, where $M_B[\omega_1,\omega_2]$ is defined in \eqref{Matrix-Def-2}. Thus, $f_{12}$ is strictly concave and the critical point of $f_{12}$ is unique. Note that $(\omega_1,\omega_2)$ is a solution of the corresponding subsystem, and so the point $(1,1)$ is a critical point of $f_{12}$. Therefore, we know that $(1,1)$ is the unique critical point of $f_{12}$ and is the global maximum point, which yields that $\max_{t_{1},t_{2}\geq 0}f_{12}(t_{1},t_{2})=f_{12}(1,1)=A_{12}$.

Based on the above arguments, we also know that
\begin{equation}
A_{12}=\max_{t_{1},t_{2}\geq0}f_{12}(t_{1},t_{2})>\max\left\{\max_{t_{1}>0}f_{12}(t_{1},0), \max_{t_{2}>0}f_{12}(0,t_{2})\right\}=\max\{A_1, A_2\}.
\end{equation}
\end{proof}

Denote
\begin{equation}\label{tem529-8}
\beta_4:=\min\left\{\beta_2, ~\beta_3, ~\frac{S^2}{16(A_1+A_2+A_3)} \right\},
\end{equation}
where $S$ is defined in \eqref{Constant-1125}, and $A_i$ is defined in \eqref{tem529-1}.

\begin{lemma}\label{Energyestimates4}
Assume that $N=4$ and $\beta_{ij}\in (0, \beta_4),i\neq j$. Then we have
\begin{equation*}
A <\min\{ A_{12}+A_3, A_{13}+A_2, A_{23}+A_1\},
\end{equation*}
\begin{equation*}
A < A_1+A_2+A_3.
\end{equation*}
\end{lemma}
\begin{proof}[\bf Proof]
Without loss of generality, we only show that $A<A_{12}+A_3$. Assume that $I_{12}(\omega^{12}_1,\omega^{12}_2)=A_{12}$ and consider
\begin{align*}
g(t_1,t_2,t_3):&=I(\sqrt{t_1}\omega^{12}_1,\sqrt{t_2}\omega^{12}_2,\sqrt{t_3}\omega_3)\\
& =\frac{1}{2}\sum_{i=1}^2t_i\|\omega^{12}_i\|_i^2+\frac{1}{2}t_3\|\omega_3\|_3^2-\frac{1}{4}\sum_{i=1}^2t_i^{2}
\int_{\mathbb{R}^4}|\omega^{12}_i|^{4}-\frac{1}{4}t_3^{2}\int_{\mathbb{R}^4}|\omega_3|^{4}\\
&\quad-\frac{1}{2}t_1t_2\int_{\mathbb{R}^4}\beta_{12}|\omega^{12}_1|^2|\omega^{12}_2|^2-\frac{1}{2}t_1t_3
\int_{\mathbb{R}^4}\beta_{13}|\omega^{12}_1|^2|\omega_3|^2
-\frac{1}{2}t_2t_3\int_{\mathbb{R}^4}\beta_{23}|\omega^{12}_2|^2|\omega_3|^2 \\
& \leq f_{12}(t_1,t_2)+f_3(t_3),
\end{align*}
where
\begin{equation}
f_3(t_3):=\frac{1}{2}t_3\|\omega_3\|_3^2-\frac{1}{4}t_3^{2}\int_{\mathbb{R}^4}|\omega_3|^{4}.
\end{equation}
It is easy to see that $\max_{t_{3}>0}f_{3}(t_{3})=A_3$. By Lemma \ref{1021-7} we have
$\max_{t_{1},t_{2}\geq 0}f_{12}(t_{1},t_{2})=f_{12}(1,1)=A_{12}$.
Therefore,
\begin{equation}\label{tab1123-1}
\max_{t_{1},t_{2},t_3\geq0}I(\sqrt{t_1}\omega^{12}_1,\sqrt{t_2}\omega^{12}_2,\sqrt{t_3}\omega_3)=
\max_{t_{1},t_{2},t_3\geq0}g(t_{1},t_{2},t_3)<A_{12}+A_3.
\end{equation}
By \eqref{1019-8} we know that $A_{12}<A_1+A_2$, then
\begin{equation}\label{tab1123-1-2}
\max_{t_{1},t_{2},t_3\geq 0}I(\sqrt{t_1}\omega^{12}_1,\sqrt{t_2}\omega^{12}_2,\sqrt{t_3}\omega_3)<A_{12}+A_3<A_1+A_2+A_3.
\end{equation}

To obtain the desired result, it suffices to show that there exist $\overline{t}_1,\overline{t}_2,\overline{t}_3>0$ such that $$(\sqrt{\overline{t}_1}\omega^{12}_1,\sqrt{\overline{t}_2}\omega^{12}_2,\sqrt{\overline{t}_3}\omega_3)\in\mathcal{M}.$$ If so,
it follows from \eqref{tab1123-1} that
\begin{equation}
A\leq \left(\sqrt{\overline{t}_1}\omega^{12}_1,\sqrt{\overline{t}_2}\omega^{12}_2,\sqrt{\overline{t}_3}\omega_3\right)\leq
\max_{t_{1},t_{2},t_3\geq0}I\left(\sqrt{t_1}\omega^{12}_1,\sqrt{t_2}\omega^{12}_2,\sqrt{t_3}\omega_3\right)<A_{12}+A_3.
\end{equation}
Thanks to \eqref{1019-8}, $A_{12}<A_1+A_2$ and then $A<A_1+A_2+A_3$. The following proof is similar to that of Lemma \ref{1021-7}.

{\bf Step 1.} Similarly as above, $g$ has a global maximum point $(\overline{t}_1,\overline{t}_2,\overline{t}_3)$ in $\overline{\mathbb{R}^{3}_{+}}$. Note that $\partial_{i}g(\overline{t}_1,\overline{t}_2,\overline{t}_3)\leq 0$ if $\overline{t}_i=0$ and $\partial_{i}g(\overline{t}_1,\overline{t}_2,\overline{t}_3)= 0$ if $\overline{t}_i>0$.
Thus,
\begin{equation}\label{1123-2-1}
\overline{t}_i\|\omega^{12}_i\|_i^2=\sum_{k=1}^3M_B[\omega^{12}_1,\omega^{12}_2,\omega_3]_{ik}\overline{t}_i\overline{t}_k, ~i=1,2,
\end{equation}
\begin{equation}\label{1123-2-2}
\overline{t}_3\|\omega_3\|_3^2=\sum_{k=1}^3M_B[\omega^{12}_1,\omega^{12}_2,\omega_3]_{3k}\overline{t}_3\overline{t}_k.
\end{equation}
It is easy to see that the above two equations are still true when $\overline{t}_i=0$. Then \eqref{1123-2-1} and \eqref{1123-2-2} imply that
\begin{equation}\label{tab1123-3}
I\left(\sqrt{\overline{t}_1}\omega^{12}_1,\sqrt{\overline{t}_2}\omega^{12}_2,\sqrt{\overline{t}_3}\omega_3\right)=
\frac{1}{4}\left(\overline{t}_1\|\omega^{12}_1\|_1^2+\overline{t}_2\|\omega^{12}_2\|_2^2+\overline{t}_3\|\omega_3\|_3^2\right)
\end{equation}
Combining this with \eqref{tab1123-1-2} we have
\begin{equation}\label{tab1123-2}
\overline{t}_1\|\omega^{12}_1\|_1^2+\overline{t}_2\|\omega^{12}_2\|_2^2+\overline{t}_3\|\omega_3\|_3^2\leq 4(A_1+A_2+A_3).
\end{equation}

{\bf Step 2.} $\overline{t}_i>0, i=1,2,3$. By contradiction, we assume that $\overline{t}_1=0$, and $(0, \overline{t}_2, \overline{t}_3)$ is the maximum point. Consider the following function
\begin{align}\label{tem1125-2}
h(s):&=I(\sqrt{s}\omega^{12}_1,\sqrt{\overline{t}_2}\omega^{12}_2,\sqrt{\overline{t}_3}\omega_3) \nonumber\\
   &=\frac{1}{2}s\|\omega^{12}_1\|_1^2-\frac{1}{2}\sum_{i=2}^{3}M_B[\omega^{12}_1,\omega^{12}_2,\omega_3]_{1i}s\overline{t}_i
-\frac{1}{4}s^2\int_{\mathbb{R}^4}|\omega^{12}_1|^{4}\nonumber\\
&+ \frac{1}{2}\overline{t}_2\|\omega^{12}_2\|_2^2+\frac{1}{2}\overline{t}_3\|\omega_3\|_3^2-\frac{1}{4}\sum_{i,j=2}^3
M_B[\omega^{12}_1,\omega^{12}_2,\omega_3]_{ij}\overline{t}_i\overline{t}_j.
\end{align}
On the other hand, by \eqref{tab1123-2} we see that
\begin{align}
\sum_{i=2}^{3}M_B[\omega^{12}_1,\omega^{12}_2,\omega_3]_{1i}\overline{t}_i & =\overline{t}_2\int_{\mathbb{R}^4}
\beta_{12}|\omega^{12}_1|^{2}|\omega^{12}_2|^{2}+\overline{t}_3\int_{\mathbb{R}^4}
\beta_{13}|\omega^{12}_1|^{2}|\omega_3|^{3}\nonumber\\
& \leq \frac{\beta_3}{S^2}\overline{t}_2\|\omega^{12}_1\|_1^2\|\omega^{12}_2\|_2^2
+\frac{\beta_3}{S^2}\overline{t}_3\|\omega^{12}_1\|_1^2\|\omega_3\|_3^2\nonumber\\
& \leq 4\frac{\beta_3}{S^2}\|\omega^{12}_1\|_1^2(A_1+A_2+A_3)\leq \frac{1}{4}\|\omega^{12}_1\|_1^2.
\end{align}
Thus,
\begin{align}\label{tem1125-3}
\frac{1}{2}s&\|\omega^{12}_1\|_1^2-\frac{1}{2}\sum_{i=2}^{3}M_B[\omega^{12}_1,\omega^{12}_2,\omega_3]_{1i}s\overline{t}_i
-\frac{1}{4}s^2\int_{\mathbb{R}^4}|\omega^{12}_1|^{4}\nonumber\\
&=\frac{1}{2}s\left(\|\omega^{12}_1\|_1^2-\sum_{i=2}^{3}M_B[\omega^{12}_1,\omega^{12}_2,\omega_3]_{1i}\overline{t}_i
-\frac{1}{4}s\int_{\mathbb{R}^4}|\omega^{12}_1|^{4} \right)\nonumber\\
& \geq \frac{1}{2}s\left(\frac{3}{4}\|\omega^{12}_1\|_1^2-\frac{1}{4}s\int_{\mathbb{R}^4}|\omega^{12}_1|^{4} \right)>0,
\end{align}
when $s>0$ small enough. Therefore, by \eqref{tem1125-2} and \eqref{tem1125-3} we have
\begin{equation}
I\left(\sqrt{s}\omega^{12}_1,\sqrt{\overline{t}_2}\omega^{12}_2,\sqrt{\overline{t}_3}\omega_3\right)>
I\left(0,\sqrt{\overline{t}_2}\omega^{12}_2,\sqrt{\overline{t}_3}\omega_3\right),
\end{equation}
which is a contradiction.
\end{proof}

\begin{lemma}\label{estimation3}
Suppose that $\beta_{ij}\in (0, \beta_4)$  and let $(u_1^n,u_2^n,u_3^n)\in \mathcal{M}$ be a minimizing sequence of $A$ with $I(u_1^n,u_2^n,u_3^n)\leq 2\overline{C}$, and $(u_1^n,u_2^n,u_3^n)\rightharpoonup (0,0,0)$ weakly in $\mathbb{D}$. Then for any $\sigma>0$ and $\delta\in (-\sigma,0)\cup (0,\sigma)$, there exists $\rho\in (-\delta,0)\cup (0,\delta)$ such that, up to a subsequence,
\begin{equation}\label{estimation3-1}
\text{ either } ~\sum_{i=1}^3\int_{B_{\sigma+\rho}}|\nabla u_i^n|^2\rightarrow 0 ~\text{ or }
~\sum_{i=1}^3\int_{\mathbb{R}^N \backslash B_{\sigma+\rho}}|\nabla u_i^n|^2\rightarrow 0.
\end{equation}
\end{lemma}

\begin{proof}[\bf Proof]
Without loss of generality, we only consider the case $\delta\in (0,\sigma)$. Since $(u_1^n,u_2^n,u_3^n)\in \mathcal{M}$ is a minimizing sequence of $A$, then $(u_1^n,u_2^n,u_3^n)$ is uniformly bounded in $\mathbb{D}$.  By Lemma \ref{PS-sequence}, $I'(u_1^n,u_2^n,u_3^n)\rightarrow 0$ as $n\rightarrow \infty$. In the following, we adopt some idea in \cite{Chen-Zou2015,Terracini1996} to prove \eqref{estimation3-1}.

Let $\mathbb{S}$ be the unit sphere of $\R^N$. Since
\begin{equation}
\int_{\sigma}^{\sigma+\delta}\mathrm{d}\rho\int_{\rho\mathbb{S}}\sum_{i=1}^3|\nabla u_i^n|^2=\int_{\sigma\leq |x|\leq\sigma+\delta}\sum_{i=1}^3|\nabla u_i^n|^2,
\end{equation}
there exists $\rho\in (0,\delta)$ such that
\begin{equation}
\int_{(\sigma+\rho)\mathbb{S}}\sum_{i=1}^3|\nabla u_i^n|^2 \leq \frac{3}{\delta}\int_{\sigma\leq |x|\leq\sigma+\delta}\sum_{i=1}^3|\nabla u_i^n|^2
\end{equation}
holds for infinitely many $n$. Therefore, $u_i^n$ is uniformly bounded in $H^1((\sigma+\rho)\mathbb{S})$, $i=1,2,3$. Note that $H^1((\sigma+\rho)\mathbb{S})$ is compactly embedded into $H^{\frac{1}{2}}((\sigma+\rho)\mathbb{S})$. Thus, we can assume that $u_i^n\rightarrow u_i$ strongly in $H^{\frac{1}{2}}((\sigma+\rho)\mathbb{S})$, $i=1,2,3$. On the other hand, by the continuity of the embedding $H^1(B_{\sigma+\rho})\hookrightarrow H^{\frac{1}{2}}((\sigma+\rho)\mathbb{S})$, we know that $u_i=0$, $i=1,2,3$, which means that
\begin{equation}\label{1021-2-2}
u_i^n\rightarrow 0 ~\text{ strongly in } H^{\frac{1}{2}}((\sigma+\rho)\mathbb{S}), ~i=1,2,3.
\end{equation}
Let $\omega^n_{i,j}, i=1,2,3,j=1,2$, be the solutions to the following problems
\begin{equation}
\begin{cases}
\Delta \omega^n_{i,1}=0 ~~ \text{ in } B_{\sigma+\delta}\setminus B_{\sigma+\rho},\\
\omega^n_{i,1}=0 ~~  \text{ on } (\sigma+\delta)\mathbb{S},\\
\omega^n_{i,1}=u^n_1 ~~  \text{ on } (\sigma+\rho)\mathbb{S},
\end{cases}
~~
\begin{cases}
\Delta \omega^n_{i,2}=0 ~~ \text{ in } B_{\sigma+\rho}\setminus B_{\sigma-\delta},\\
\omega^n_{i,2}=0 ~~  \text{ on } (\sigma-\delta)\mathbb{S},\\
\omega^n_{i,2}=u^n_1 ~~  \text{ on } (\sigma+\rho)\mathbb{S}.
\end{cases}
\end{equation}
Combining \eqref{1021-2-2} with continuity of the inverse Laplace operator from $H^{\frac{1}{2}}(\partial\Omega)$ to $H^1(\Omega)$, we see that $\omega^n_{i,1}\rightarrow 0$ strongly in $H^1(B_{\sigma+\delta}\backslash B_{\sigma+\rho})$
and $\omega^n_{i,2}\rightarrow 0$ strongly in $H^1(B_{\sigma+\rho}\backslash B_{\sigma-\delta})$, i=1, 2, 3. For any $i=1, 2, 3$, define
\begin{equation}
u_{i,1}^n(x)=
\begin{cases}
u_i^n(x)  ~~ \text{ if } x\in B_{\sigma+\rho},\\
\omega_{i,1}^n(x)  ~~ \text{ if } x\in B_{\sigma+\delta}\setminus B_{\sigma+\rho},\\
0 ~~\text{ elsewhere},
\end{cases}
~~u_{i,2}^n(x)=
\begin{cases}
0  ~~ \text{ if } x\in B_{\sigma-\delta},\\
\omega_{i,2}^n(x)  ~~ \text{ if } x\in B_{\sigma+\rho}\setminus B_{\sigma-\delta},\\
u_i^n(x) ~~\text{ elsewhere}.
\end{cases}
\end{equation}
Obviously,
\begin{equation}
\|u_i^n\|_i^2= \|u_{i,1}^n\|_i^2+\|u_{i,2}^n\|_i^2+o(1).
\end{equation}
Moreover,
\begin{equation}\label{26-1}
I'(u_{1,1}^n,u_{2,1}^n,u_{3,1}^n)\mathbf{u}^n_{1,i}=0, \quad I'(u_{1,2}^n,u_{2,2}^n,u_{3,2}^n)\mathbf{u}^n_{2,i}=0,
\end{equation}
where
\begin{equation}
(\mathbf{u}^n_{1,i})_j=0 ~\text{ when } j\neq i, \quad (\mathbf{u}^n_{1,i})_j=u_{i,1}^n ~\text{ when } j= i,
\end{equation}
\begin{equation}
(\mathbf{u}^n_{2,i})_j=0 ~\text{ when } j\neq i, \quad (\mathbf{u}^n_{2,i})_j=u_{i,2}^n ~\text{ when } j= i.
\end{equation}

Next, we claim that
\begin{equation}\label{26-3}
\text{ either } \lim_{n\rightarrow\infty}\sum_{i=1}^3\|u_{i,1}^n\|^2=0 \text{ or }
\lim_{n\rightarrow\infty}\sum_{i=1}^3\|u_{i,2}^n\|^2=0.
\end{equation}
By contradiction, if \eqref{26-3} does not hold, then up to a subsequence,
\begin{equation}\label{26-4}
\text{ both } \sum_{i=1}^3\lim_{n\rightarrow\infty}\|u_{i,1}^n\|^2>0 \text{ and }
\sum_{i=1}^3\lim_{n\rightarrow\infty}\|u_{i,2}^n\|^2>0.
\end{equation}
Set
\begin{equation}
a_i:=\lim_{n\rightarrow\infty}\|u_{i,1}^n\|^2, \quad \mathbf{a}:=(a_1, a_2, a_3).
\end{equation}
It follows from \eqref{26-4} that $\sum_{i=1}a_i>0$, which implies that there exists at least one component of $\mathbf{a}$ is positive. Then we consider the following several cases.\\

{\bf Case 1.} Three components of $\mathbf{a}$ are positive, that is, $a_i=\lim_{n\rightarrow \infty}\|u_{i,1}^n\|^2>0$ for any $i=1, 2, 3$.

Note that norms $\|\cdot\|_{i}, i=1,2,3$ are equivalent to $\|\cdot\|$. By \eqref{26-1}, for any $i=1,2,3$ we have
\begin{equation}\label{26-5}
\|u_{i,1}^n\|_i^2=\sum_{j=1}^3\beta_{ij}\int_{\R^4}|u_{i,1}^n|^{2}|u_{j,1}^n|^{2}+o(1),
\end{equation}
which yields that
\begin{equation*}
\liminf_{n\rightarrow \infty}|u_{i,1}^n|_{4}^{4}>0, \text{ for any } i=1,2,3.
\end{equation*}
We claim that there exist $s_i^n>0$,$i=1,2,3$ such that
\begin{equation}\label{tem0511-3}
 \left(s_1^nu_{1,1}^n,s_2^nu_{2,1}^n,s_3^nu_{3,1}^n\right)\in \mathcal{M} \text{ and }  \lim_{n\rightarrow \infty}s_i^n=1.
\end{equation}
In order to prove this claim, we consider the following system
\begin{equation}\label{tem1125-6}
\sum_{j=1}^3M_B[u_{1,1}^n,u_{2,1}^n,u_{3,1}^n]_{ij}s_j=\|u_{i,1}^n\|_i^2, ~i=1,2,3.
\end{equation}
By the proof of Lemma \ref{PS-sequence} we see that
\begin{equation}\label{tem1127-1}
M_B[u_{1,1}^n,u_{2,1}^n,u_{3,1}^n]_{ii}-\sum_{j\neq i}\left|M_B[u_{1,1}^n,u_{2,1}^n,u_{3,1}^n]_{ij}\right|\geq C, ~i=1,2,3,
\end{equation}
which implies that the matrix $M_B[u_{1,1}^n,u_{2,1}^n,u_{3,1}^n]$ is strictly diagonally dominant.
Thus, by the Gershgorin circle theorem we know that the matrix $M_B[u_{1,1}^n,u_{2,1}^n,u_{3,1}^n]$ is positive definite, and
\begin{equation}\label{tem1127-2}
\det(M_B[u_{1,1}^n,u_{2,1}^n,u_{3,1}^n])\geq C,
\end{equation}
where the constant $C$ is independent on $n$.
Therefore system \eqref{tem1125-6} has a unique solution $(s_1^n, s_2^n, s_3^n)$, and so $\left(s_1^nu_{1,1}^n,s_2^nu_{2,1}^n,s_3^nu_{3,1}^n\right)\in \mathcal{M}$.
Moreover, by Cramer's rule we know that $s_i^n$ are uniformly bounded. Set $\overline{s}_i:=\lim_{n\rightarrow\infty}s_i^n$,
and
\begin{equation}\label{tem1127-3}
b_{ij}:=\lim_{n\rightarrow\infty}M_B[u_{1,1}^n,u_{2,1}^n,u_{3,1}^n]_{ij}.
\end{equation}
It follows from \eqref{tem1127-1} that
\begin{equation}
b_{ii}-\sum_{j\neq i}\left|b_{ij} \right|\geq C.
\end{equation}
Hence, by the Gershgorin circle theorem we see that the matrix $B$ is positive definite and $\det(B)\neq 0$, where
\begin{equation}
B:= (b_{ij})_{1\leq i,j\leq 3}.
\end{equation}
By \eqref{26-5} and \eqref{tem1125-6} we have
\begin{equation}\label{tem0511-5}
\sum_{j=1}^3b_{ij}(s_j-1)=0, ~i=1,2,3.
\end{equation}
We deduce form $\det(B)\neq 0$ that $\overline{s}_i=1, i=1,2,3$. So, \eqref{tem0511-3} is true. Hence,
\begin{align*}
A & = \lim_{n\rightarrow \infty} I(\mathbf{u}_n)=\frac{1}{4}\lim_{n\rightarrow \infty}(\|u_1^n\|_1^2+\|u_2^n\|_2^2+\|u_3^n\|_3^2)  \\
   & = \frac{1}{4}\lim_{n\rightarrow \infty}\sum_{j=i}^3|s_i^n|^2\|u_{i,1}^n\|_i^2+\frac{1}{4}\lim_{n\rightarrow \infty}\sum_{j=i}^3\|u_{i,2}^n\|_i^2\\
   & > \frac{1}{4}\lim_{n\rightarrow \infty}\sum_{j=i}^3|s_i^n|^2\|u_{i,1}^n\|_i^2\geq A,
\end{align*}
which is a contradiction. Therefore, Case 1 is impossible.\\

{\bf Case 2.} Only two components of $\mathbf{a}$ are positive. Without loss of generality, we assume that $a_1>0, a_2>0, a_3=0$, that is $\lim_{n\rightarrow \infty}\|u_{1,1}^n\|^2>0$, $\lim_{n\rightarrow \infty}\|u_{2,1}^n\|^2>0$, $\lim_{n\rightarrow \infty}\|u_{3,1}^n\|^2=0$.

Based on Lemma \ref{Estimate2}, it is standard to see that $\lim_{n\rightarrow \infty}\|u_3^n\|^2>0$. Then we have
\begin{equation}
\lim_{n\rightarrow \infty}\|u_{3,2}^n\|^2=\left(\lim_{n\rightarrow \infty}\|u_3^n\|^2-\lim_{n\rightarrow \infty}\|u_{3,1}^n\|^2\right)>0.
\end{equation}

{\bf Case 2.1} $\lim_{n\rightarrow \infty}\|u_{1,2}^n\|^2=0$ and $\lim_{n\rightarrow \infty}\|u_{2,2}^n\|^2=0$.

By \eqref{26-1} we get that
\begin{equation}
\|u_{3,2}^n\|_3^2= \int_{\R^4}|u_{3,2}^n|^{4}+\beta_{31}\int_{\R^4}|u_{3,2}^n|^{2}|u_{1,2}^n|^{2}
+\beta_{32}\int_{\R^4}|u_{3,2}^n|^{2}|u_{2,2}^n|^{2} +o(1),
\end{equation}
which implies that
\begin{equation}\label{1019-1}
\lim_{n\rightarrow \infty}\|u_{3,2}^n\|_3^2\geq S_3^{2}.
\end{equation}
Note that $\lim_{n\rightarrow \infty}\|u_{3,1}^n\|^2=0$, then we deduce from \eqref{26-1} that
\begin{equation}
\|u_{1,1}^n\|_1^2=\int_{\R^4}|u_{1,1}^n|^{4}+\beta_{21}\int_{\R^4}|u_{1,1}^n|^{2}|u_{2,1}^n|^{2}+o(1),
\end{equation}
\begin{equation}
\|u_{2,1}^n\|_2^2=\int_{\R^4}|u_{2,1}^n|^{4}+\beta_{21}\int_{\R^4}|u_{1,1}^n|^{2}|u_{2,1}^n|^{2}+o(1).
\end{equation}
Since $\beta_{ij}<1$ for any $i\neq j$, then it is easy to see that
there exist $a_1^n,a_2^n$ such tat
\begin{equation}
(a_1^nu_{1,1}^n, a_2^nu_{2,1}^n)\in \mathcal{M}_{12} ~\text{ and } \lim_{n\rightarrow \infty}a_i^n=1,
\end{equation}
where $\mathcal{M}_{12}$ is defined in \eqref{tem107-4}.
Thus,
\begin{align}\label{1019-2}
A_{12} & \leq \lim_{n\rightarrow \infty}I_{12}(a_1^nu_{1,1}^n, a_2^nu_{2,1}^n)=\frac{1}{4}\lim_{n\rightarrow \infty}
(|a_1^n|^2\|u_{1,1}^n\|_1^2+|a_2^n|^2\|u_{2,1}^n\|_2^2) \nonumber\\
   & = \frac{1}{4}\lim_{n\rightarrow \infty}
(\|u_{1,1}^n\|_1^2+\|u_{2,1}^n\|_2^2)
\end{align}
It follows from \eqref{1019-1} and \eqref{1019-2} that
\begin{align}\label{26-8}
A & = \lim_{n\rightarrow \infty} I(\mathbf{u}_n)=\frac{1}{4}\lim_{n\rightarrow \infty}(\|u_1^n\|_1^2+\|u_2^n\|_2^2+\|u_3^n\|_3^2)  \nonumber\\
&\geq \frac{1}{4}\lim_{n\rightarrow \infty}
(|a_1^n|^2\|u_{1,1}^n\|_1^2+|a_2^n|^2\|u_{2,1}^n\|_2^2+\|u_{3,2}^n\|_3^2) \nonumber\\
&\geq A_{12}+A_3,
\end{align}
which contradicts with Lemma \ref{Energyestimates4}. Therefore, Case 2.1 is impossible.\\

{\bf Case 2.2} Either $\lim_{n\rightarrow \infty}\|u_{1,2}^n\|^2=0$ or $\lim_{n\rightarrow \infty}\|u_{2,2}^n\|^2=0$. Without loss of generality, we assume that $\lim_{n\rightarrow \infty}\|u_{1,2}^n\|^2=0$ and $\lim_{n\rightarrow \infty}\|u_{2,2}^n\|^2>0$.

Note that $\lim_{n\rightarrow \infty}\|u_{1,2}^n\|^2=0$, combining this with \eqref{26-1}
we know that there exist $\widehat{s}_2^n,\widehat{s}_3^n$ such tat
\begin{equation}
(\widehat{s}_2^nu_{2,2}^n, \widehat{s}_3^nu_{3,2}^n)\in \mathcal{M}_{23} ~\text{ and } \lim_{n\rightarrow \infty}\widehat{s}_i^n=1.
\end{equation}
Thus,
\begin{align}\label{26-9}
A_{23}  \leq \lim_{n\rightarrow \infty}I_{23}(\widehat{s}_2^nu_{2,2}^n, \widehat{s}_3^nu_{3,2}^n)
   = \frac{1}{4}\lim_{n\rightarrow \infty}
(\|u_{2,2}^n\|_2^2+\|u_{3,2}^n\|_3^2).
\end{align}
Combining this with Lemma \ref{1021-7} and \eqref{1019-2} we get that
\begin{align}\label{26-10}
A & = \lim_{n\rightarrow \infty} I(\mathbf{u}_n)=\frac{1}{4}\lim_{n\rightarrow \infty}(\|u_1^n\|_1^2+\|u_2^n\|_2^2+\|u_3^n\|_3^2)  \nonumber\\
&\geq \frac{1}{4}\lim_{n\rightarrow \infty}
\left(a_1^n\|u_{1,1}^n\|_1^2+a_2^n\|u_{2,1}^n\|_2^2\right)+\frac{1}{4}\lim_{n\rightarrow \infty}\left(\widehat{s}_2^n\|u_{2,2}^n\|_2^2+\widehat{s}_3^n\|u_{3,2}^n\|_3^2\right) \nonumber\\
&\geq A_{12}+A_{23}>A_{12}+A_3,
\end{align}
which contradicts with Lemma \ref{Energyestimates4}. Therefore, {\bf Case 2.2} is impossible.\\

{\bf Case 3.} Only one component of $\mathbf{a}$ is positive. Without loss of generality, we assume that $a_1=0, a_2=0, a_3>0$, that is $\lim_{n\rightarrow \infty}\|u_{1,1}^n\|^2=0$, $\lim_{n\rightarrow \infty}\|u_{2,1}^n\|^2=0$, $\lim_{n\rightarrow \infty}\|u_{3,1}^n\|^2>0$.

It is easy to see that
\begin{equation}
\lim_{n\rightarrow \infty}\|u_{1,2}^n\|^2>0 \text{ and } \lim_{n\rightarrow \infty}\|u_{2,2}^n\|^2>0.
\end{equation}
Similarly to \eqref{1019-1} we have
\begin{equation}\label{26-12}
\lim_{n\rightarrow \infty}\|u_{3,1}^n\|_3^2\geq S_3^{2}.
\end{equation}
If $\lim_{n\rightarrow \infty}\|u_{3,2}^n\|^2>0$, then similarly to case 1 we can get a contradiction. If $\lim_{n\rightarrow \infty}\|u_{3,2}^n\|^2=0$,
it is easy to get that there exist $\widehat{a}_1^n, \widehat{a}_2^n$ such tat
\begin{equation}
(\widehat{a}_1^nu_{1,2}^n, \widehat{a}_2^nu_{2,2}^n)\in \mathcal{M}_{12} ~\text{ and } \lim_{n\rightarrow \infty}\widehat{a}_i^n=1,
\end{equation}
and so
\begin{equation}
\frac{1}{4}\lim_{n\rightarrow \infty}
(\|u_{1,1}^n\|_1^2+\|u_{2,1}^n\|_2^2)\geq A_{12}.
\end{equation}
Combining this with \eqref{1019-1} we have
\begin{equation}
A = \lim_{n\rightarrow \infty} I(\mathbf{u}_n)=\frac{1}{4}\lim_{n\rightarrow \infty}(\|u_1^n\|_1^2+\|u_2^n\|_2^2+\|u_3^n\|_3^2)\geq A_{12}+A_3,
\end{equation}
which contradicts with Lemma \ref{Energyestimates4}. Therefore, {\bf Case 3} is impossible.

From the above arguments,  Case 1-3 are impossible, then \eqref{26-4} is not true.
Therefore, \eqref{26-3} holds true and so does for \eqref{estimation3-1} thanks to the definition of $u_{i,j}$.
\end{proof}

\subsubsection{\bf Proof of Theorem \ref{thm2}-(1)}\label{the proof}

Set
\begin{equation}\label{1031-1}
\kappa:=\frac{1}{2}\min_{1\leq i\leq3}\{\mathcal{S}^2-4A_i\}.
\end{equation}
It is easy to see that $\mathcal{S}^2-\kappa>4A_i$, for any $i=1,2,3$.
Denote
\begin{equation}\label{tem529-12}
\beta_{5}:=\frac{\kappa S}{6\mathcal{S}^3}, \quad
\widetilde{\beta}:=\min\{ \beta_4, \beta_{5}\},
\end{equation}
where $\beta_4$ is defined in \eqref{tem529-8}. From now on, we assume that $\beta_{ij}\in (0, \widetilde{\beta})$ and fix such an $\beta_{ij}$.

\begin{proof}[\bf The proof of Theorem \ref{thm2}-(1)]
Let $\widetilde{\mathbf{u}}^n=(\widetilde{u}_1^n, \widetilde{u}_2^n,\widetilde{u}_3^n)$ be a minimizing sequence for the level $A$. Denote
\begin{equation}
G(u_1,u_2,u_3):=\sum_{i=1}^3\left(|\nabla u_i|^2-\frac{\lambda_i}{|x|^2}\right).
\end{equation}
It is easy to see that there exists $a_n>0$ such that
\begin{equation}\label{thm1-1}
\int_{B_{a_n}}G(\widetilde{\mathbf{u}}^n)=\int_{\R^4\backslash B_{a_n}}G(\widetilde{\mathbf{u}}^n)=
\frac{1}{2}\sum_{i=1}^3\|\widetilde{u}_i^n\|_i^2.
\end{equation}
Define
\begin{equation}
\widehat{u}_i^n:=a_n^{\frac{N-2}{2}}\widetilde{u}_i^n(a_nx), i=1,2,3, \text{ and }
\widehat{\mathbf{u}}^n:=(\widehat{u}_1^n,\widehat{u}_2^n,\widehat{u}_3^n).
\end{equation}
Then we have $\widehat{\mathbf{u}}^n\in \mathcal{M}$ and $I(\widehat{\mathbf{u}}^n)\rightarrow A$.
Moreover, by \eqref{thm1-1} we know that
\begin{equation}\label{1019-4}
\int_{B_1}G(\widehat{\mathbf{u}}^n)=\int_{\R^4\backslash B_1}G(\widehat{\mathbf{u}}^n)=
\frac{1}{2}\sum_{i=1}^3\|\widehat{u}_i^n\|_i^2\rightarrow 2A.
\end{equation}
By Lemma \ref{PS-sequence} we see that
\begin{equation}
\lim_{n\rightarrow \infty}I(\mathbf{u}^n)=A  \text{ and } \lim_{n\rightarrow \infty}I'(\mathbf{u}^n)=0.
\end{equation}
We deduce from \eqref{1019-4} that
\begin{equation}
\lim_{n\rightarrow \infty}\int_{B_1}G(\mathbf{u}^n)=\lim_{n\rightarrow \infty}\int_{B_1}G(\widehat{\mathbf{u}}^n)=2A,
\quad
\lim_{n\rightarrow \infty}\int_{\R^4\backslash B_1}G(\mathbf{u}^n)=\int_{\R^N\backslash B_1}G(\widehat{\mathbf{u}}^n)=2A.
\end{equation}

Observe that $\{(u_1^n,u_2^n,u_3^n)\}$ are uniformly bounded in $\mathbf{D}$, then up to a subsequence we suppose that
$(u_1^n,u_2^n,u_3^n)\rightharpoonup (u_1,u_2,u_3)$ weakly in $\mathbf{D}$. Hence, $I'(u_1,u_2,u_3)=0$. Next, we will present that $u_i\neq 0$ for any $i=1,2,3$ and $\mathbf{u}=(u_1,u_2,u_3)\in \mathcal{M}$ by contradiction.

{\bf Case 1}. $(u_1,u_2,u_3)=(0,0,0)$.\\
Based upon Lemma \ref{estimation3}, following the arguments in the proof of \cite[Theorem 1.2]{Zou 2015}, we can get that
\begin{equation}
\int_{\R^4}\frac{|u_i^n|^2}{|x|^2}=o(1), ~\text{ for any } i=1,2,3.
\end{equation}
Combining this with $(u_1^n,u_2^n,u_3^n)\in \mathcal{M}$ we have

\begin{align}
 \int_{\R^4}|\nabla u_i^n|^2 & =\int_{\R^4}|u_i^n|^4+ \sum_{j\neq i}\int_{\R^4}\beta_{ij}|u_i^n|^2|u_j^n|^2+o(1)  \\
   & \leq \mathcal{S}^{-2}\left(\int_{\R^4}|\nabla u_i^n|^2\right)^2+\sum_{j\neq i}\frac{\beta}{\mathcal{S}S}\|u_j^n\|_j^2\int_{\R^4}|\nabla u_i^n|^2+o(1)\\
   & \leq \mathcal{S}^{-2}\left(\int_{\R^4}|\nabla u_i^n|^2\right)^2+ \beta\frac{6\mathcal{S}}{S}\int_{\R^4}|\nabla u_i^n|^2+o(1),
\end{align}
which yields that
\begin{equation*}
\lim_{n\rightarrow \infty}\int_{\R^4}|\nabla u_i^n|^2\geq \mathcal{S}^2-\beta\frac{6\mathcal{S}^3}{S}\geq\mathcal{S}^2-\kappa>4A_i, ~\text{ for any } i=1,2,3.
\end{equation*}
Therefore,
\begin{equation*}
  A = \lim_{n\rightarrow \infty}\frac{1}{4}\sum_{i=1}^3\|u_i^n\|_i^2
   = \lim_{n\rightarrow \infty}\frac{1}{4}\sum_{i=1}^3\int_{\R^4}|\nabla u_i^n|^2
   \geq \sum_{i=1}^3A_i,
\end{equation*}
which contradicts with Lemma \ref{Energyestimates4}.

{\bf Case 2}. Only one component of $\mathbf{u}$ is not zero.\\
Without loss of generality, we may assume that $u_1\neq 0, u_2=0, u_3=0$. Since $I'(u_1,u_2,u_3)(u_1,0,0)=0$, then
\begin{equation}
\|u_1\|_1^2=\int_{\R^4}|u_1|^4\leq S_1^{-2}\|u_1\|_1^4,
\end{equation}
and so
\begin{equation}
\|u_1\|_1^2\geq S_1^{2}.
\end{equation}

{\bf Case 2.1}. $u_1^n\rightarrow u_1$ strongly in $D^{1,2}(\R^4)$.\\
Note that $u_1^n\rightarrow u_1$ strongly in $L^4(\R^4)$. Note that $u_i^n\rightharpoonup 0$ weakly in $D^{1,2}(\R^4)$, $i=2,3$, then we see that $(u_i^n)^2\rightharpoonup 0$ weakly in $L^2(\R^4)$. Hence,
\begin{equation}
\int_{\R^4}|u_1^n|^2|u_i^n|^2=o(1), ~\text{ for any } i=2,3.
\end{equation}
Similarly to the proof of Lemma \ref{estimation3} we know that there exist $\overline{s}_2^n,\overline{s}_3^n$ such tat
\begin{equation}
(\sqrt{\overline{s}_2^n}u_{2}^n, \sqrt{\overline{s}_3^n}u_{3}^n)\in \mathcal{M}_{23} ~\text{ and } \lim_{n\rightarrow \infty}\overline{s}_i^n=1.
\end{equation}
Therefore,
\begin{align}
A & = \lim_{n\rightarrow \infty} I(\mathbf{u}_n) \geq \frac{1}{4}\|u_1\|_1^2+\frac{1}{4}\lim_{n\rightarrow \infty}\left(\|u_{2}^n\|_2^2+\|u_{3}^n\|_3^2\right) \nonumber\\
&\geq \frac{1}{4}S_1^2+\frac{1}{4}\lim_{n\rightarrow \infty}\left(\overline{s}_2^n|\|u_{2}^n\|_2^2+\overline{s}_3^n\|u_{3}^n\|_3^2\right) \nonumber\\
&\geq A_{1}+A_{23},
\end{align}

{\bf Case 2.2}. $\lim_{n\rightarrow \infty}\|u_1^n -u_1\|_1^2>0$.\\
Set $\omega_1^n=u_1^n -u_1$. Note that $(u_1^n,u_2^n,u_3^n)\in \mathcal{M}$, then by the Brezis-Lieb Lemma we get that
\begin{equation}
\|\omega_1^n\|_1^2=\sum_{j=1}M_B[\omega_{1}^n,u_{2}^n,u_{3}^n]_{1j}+o(1), ~\quad \|u_i^n\|_i^2=\sum_{j=1}M_B[\omega_{1}^n,u_{2}^n,u_{3}^n]_{ij}+o(1), i=2,3.
\end{equation}
Similarly to the proof of Lemma \ref{estimation3} there exist $\widetilde{s}_i^n>0$,$i=1,2,3$ such that
\begin{equation}
 \left(\sqrt{\widetilde{s}_1^n}\omega_{1}^n,\sqrt{\widetilde{s}_2^n}u_{2}^n,\sqrt{\widetilde{s}_3^n}u_{3}^n\right)\in \mathcal{M} \text{ and }  \lim_{n\rightarrow \infty}\widetilde{s}_i^n=1.
\end{equation}
Therefore,
\begin{align*}
A & = \lim_{n\rightarrow \infty} I(\mathbf{u}_n)= \frac{1}{4}\|u_1\|_1^2+\frac{1}{4}\lim_{n\rightarrow \infty}(\|\omega_{1}^n\|_1^2+\|u_{2}^n\|_2^2+\|u_{3}^n\|_3^2)\\
   & >\frac{1}{4}\lim_{n\rightarrow \infty}(\widetilde{s}_1^n\|\omega_{1}^n\|_1^2+\widetilde{s}_2^n\|u_{2}^n\|_2^2+\widetilde{s}_3^n\|u_{3}^n\|_3^2)\\
   & \geq A,
\end{align*}
which contradicts with Lemma \ref{Energyestimates4}.

{\bf Case 3}. Only two components of $\mathbf{u}$ are nontrivial.\\
Without loss of generality, we may assume that $u_1\neq 0, u_2\neq0, u_3=0$. Since $I'(u_1,u_2,0)=0$, then we have
\begin{equation}\label{1021-11}
I(u_1,u_2,0)\geq A_{12}.
\end{equation}

{\bf Case 3.1}. $u_1^n\rightarrow u_1$ strongly in $D^{1,2}(\R^4)$ and $u_2^n\rightarrow u_2$ strongly in $D^{1,2}(\R^4)$.\\
Since $u_3^n\rightarrow 0$ weakly in $D^{1,2}(\R^4)$, then we see that
\begin{equation}
\int_{\R^4}|u_1^n|^2|u_3^n|^2=o(1), \quad \int_{\R^4}|u_2^n|^2|u_3^n|^2=o(1).
\end{equation}
Combining this with $(u_1^n,u_2^n,u_3^n)\in \mathcal{M}$ we know that
\begin{equation}
\|u_3^n\|_3^2=\int_{\R^4}|u_{3}^n|^{4}+o(1)\leq S_3^{-2}\|u_3^n\|_3^4+o(1),
\end{equation}
which implies that
\begin{equation}\label{1031-3}
\lim_{n\rightarrow \infty}\|u_3^n\|_3^2\geq S_3^{2}.
\end{equation}
It follows from \eqref{1021-11} and \eqref{1031-3} that
\begin{align*}
A = \lim_{n\rightarrow \infty} I(\mathbf{u}_n) \geq \frac{1}{4}(\|u_1\|_1^2+\|u_2\|_2^2)+\frac{1}{4}\lim_{n\rightarrow \infty}\|u_{3}^n\|_3^2
    \geq A_{12}+A_3,
\end{align*}
which contradicts with Lemma \ref{Energyestimates4}.

{\bf Case 3.2}. $\lim_{n\rightarrow \infty}\|u_1^n -u_1\|_1^2>0$ and $u_2^n\rightarrow u_2$ strongly in $D^{1,2}(\R^4)$.\\
Set $\omega_1^n=u_1^n -u_1$. Since $u_3^n\rightarrow 0$ weakly in $D^{1,2}(\R^4)$, then we see that
\begin{equation}\label{1021-9}
\int_{\R^4}|u_2^n|^2|u_3^n|^2=o(1), \quad \int_{\R^4}|u_1^n|^2|u_2^n|^2=\int_{\R^4}|u_1|^2|u_2|^2.
\end{equation}
Therefore, by Brezis-Lieb Lemma and \eqref{1021-9} and $I'(u_1,u_2,0)(u_1,0,0)=0$ we have
\begin{equation}
\|\omega_1^n\|_1^2=\int_{\R^4}|\omega_1^n|^{4}+\beta_{13}\int_{\R^4}|\omega_1^n|^{2}|u_3^n|^{2}+o(1),\quad
\|u_3^n\|_1^2=\int_{\R^4}|u_3^n|^{4}+\beta_{13}\int_{\R^4}|\omega_1^n|^{2}|u_3^n|^{2}+o(1).
\end{equation}
Similarly to the proof of Lemma \ref{estimation3} we know that there exist $\pi_1^n,\pi_3^n>0$ such tat
\begin{equation}\label{1021-10}
\left(\sqrt{\pi_1^n}\omega_1^n, \sqrt{\pi_3^n}u_{3}^n\right)\in \mathcal{M}_{13} ~\text{ and } \lim_{n\rightarrow \infty}\pi_i^n=1.
\end{equation}
It follows from \eqref{1021-12}, \eqref{1021-11}, \eqref{1021-10} and Lemma \ref{1021-7} that
\begin{align*}
A & = \lim_{n\rightarrow \infty} I(\mathbf{u}_n)\geq \frac{1}{4}(\|u_1\|_1^2+\|u_2\|_2^2)+\frac{1}{4}\lim_{n\rightarrow \infty}(\|u_{3}^n\|_3^2+\|\omega_1^n\|_1^2)\\
   & \geq A_{12}+\frac{1}{4}\lim_{n\rightarrow \infty}(\pi_3^n\|u_{3}^n\|_3^2+\pi_1^n\|\omega_1^n\|_1^2),\\
   & \geq A_{12}+ A_{13}>A_2+A_{13},
\end{align*}
which contradicts with Lemma \ref{Energyestimates4}.
\end{proof}

\subsection{The case for $N=3$}

In this subsection, we present the proof of Theorem \ref{thm2}-(2).
Given $\Gamma\subseteq\{1,2,3\}$ with $|\Gamma|=2,3$, consider the following system
\begin{equation}\label{mainsystem9}
\begin{cases}
-\Delta u_{i}-\frac{\lambda_{i}}{|x|^2}u_{i}=|u_i|^{4}u_i+\sum_{j\neq i}\beta_{ij}|u_{j}|^{3}|u_i|u_i, ~x\in \mathbb{R}^3, \\
u_i\in D^{1,2}(\mathbb{R}^3), i \in \Gamma.
\end{cases}
\end{equation}
Recalling $I_\Gamma, A_\Gamma, \mathcal{M}_\Gamma$ are defined in \eqref{1.10-4}, \eqref{tem107-4}, \eqref{tem529-10}, then we have the following proposition.
\begin{prop}\label{1103-1}
Assume that $\lambda_i\in (0, \Lambda_3)$. There exists $\widehat{\beta}>0$ such that, if
\begin{equation}
0\leq \beta_{ij}< \widehat{\beta}, ~\text{ for any } i\neq j,
\end{equation}
then $A_\Gamma=I_\Gamma(\mathbf{u})$ and \eqref{mainsystem9} has a least energy positive solution.
\end{prop}
We mention that $\widehat{\beta}$ is defined in \eqref{tem529-4}.  When $|\Gamma|=2$, by Proposition \ref{1103-1} we have
\begin{cor}\label{1107-1}
When $N=3$, $\Gamma=\{i,j\}$ and $\lambda_i\in (0, \Lambda_3)$. There exists $\widehat{\beta}>0$ such that, if
\begin{equation}
0\leq \beta_{ij}< \widehat{\beta}, ~\text{ for any } i\neq j,
\end{equation}
then $A_\Gamma=I_\Gamma(\mathbf{u})$ and \eqref{mainsystem9} has a least energy positive solution.
\end{cor}
Noting that Theorem \ref{thm2}-(2) is the consequence of Proposition \ref{1103-1}, in the following, we turn to prove that Proposition \ref{1103-1} holds.
\subsubsection{Preliminary results}
Denote
\begin{equation}\label{tem529-5}
\widehat{\beta}_1:=\frac{7S^3}{12(6\overline{C})^2},
\end{equation}
where $S$ is defined in \eqref{Constant-1125} and $\overline{C}$ is defined in \eqref{Constant-2-1}.
\begin{lemma}\label{Estimate2-3}
Assume that $N=3$ and $0\leq\beta_{ij}< \widehat{\beta}_1$, for any $i\neq j$, then there exist $C_3,C_4>0$ dependent on $\lambda_i, S, \overline{C}$, such that for any $\mathbf{u}_\Gamma \in \mathcal{M}_\Gamma$ with $I_\Gamma(\mathbf{u}_\Gamma)\leq 2\overline{C}$ we have
\begin{equation}
C_3\leq|u_i|^6_{6}\leq C_4, ~i=1,2,3.
\end{equation}
\end{lemma}
\begin{proof}[\bf Proof]
The proof is similar to that of \cite[Lemma 2.2]{LiuYouZou2023}, so we omit it.

\end{proof}

Before proceeding, we define the matrix $G_\Gamma(\mathbf{u})=(a_{ij}(\mathbf{u}))_{i,j\in \Gamma}$ by
\begin{equation*}
a_{ii}(\mathbf{u})=4\int_{\mathbb{R}^3} |u_i|^6 +\sum_{\substack{ j\neq i}} \int_{\mathbb{R}^3} \beta_{ij}|u_i|^3|u_j|^3,\quad
a_{ij}(\mathbf{u})=3\int_{\mathbb{R}^3}\beta_{ij}|u_i|^3|u_j|^3,  ~i\neq j.
\end{equation*}
Set	
\begin{equation*}
\Lambda_\Gamma:=\{\mathbf{u}\in \mathbb{H}:G_\Gamma(\mathbf{u}) \text{ is strictly diagonally dominant}\}.
\end{equation*}

Based on Lemma \ref{Estimate2-3} and definition of $\widehat{\beta}_1$, following the proof of Lemma 2.3, Lemma 2.4, Lemma 2.5 in \cite{LiuYouZou2023}, we can get the following three lemmas, and we omit the proofs.
\begin{lemma}\label{naturalconstraint}
Assume that
$$
  0<\beta_{ij} < \widehat{\beta}_1 \quad  \forall i,j = 1,2,3, i \neq j,
$$
then the set $\mathcal{M}_\Gamma \cap\Lambda_\Gamma$ is a smooth manifold.
Moreover, the constrained critical points of $I$ on $\mathcal{M} \cap\Lambda$ are free critical points of $I$. In other words, $\mathcal{M}_\Gamma \cap\Lambda_\Gamma$ is a natural constraint.
\end{lemma}

\begin{lemma}\label{diagonally dominant}
Assume that
$$ 0<\beta_{ij} < \widehat{\beta}_1 \quad  \forall i,j = 1,2,3, i \neq j ,
$$
then we have
$$
\mathcal{M}_\Gamma \cap \{\mathbf{u}\in \mathbb{H}_\Gamma: I_\Gamma(\mathbf{u})\leq 2 \overline{C}\} \subset \Lambda_\Gamma.
$$
Moreover, the constrained critical points of $I_\Gamma$ on $\mathcal{M}_\Gamma$ satisfying $I_\Gamma(\mathbf{u})\leq 2 \overline{C}$ are free critical points of $I_\Gamma$.
\end{lemma}

\begin{lemma}\label{PS-sequence-3}
Suppose that $N=3$ and $0\leq\beta_{ij}< \widehat{\beta}_1$, for any $i\neq j$. Let $\widehat{\mathbf{u}}_{n}
\in\mathcal{M}_\Gamma$ be a minimizing sequence of $A_\Gamma$. Then there exists a sequence $\{\mathbf{u}_n\}\in\mathcal{M}_\Gamma$ such that
\begin{equation}
\lim_{n\rightarrow \infty} I_\Gamma(\mathbf{u}_n)=A_\Gamma, \quad \lim_{n\rightarrow \infty} I_\Gamma'(\mathbf{u}_n)=0.
\end{equation}
\end{lemma}

\subsubsection{Energy estimates}
In this subsection, we will establish the energy estimates.

\begin{lemma}\label{Energyestimates6}
Assume that $N=3$ and $\beta_{ij}\in (0, 1),i\neq j$. Given $\Gamma \subseteq \{1,2, 3\}$, suppose that $ A_Q $ is achieved by $\mathbf{u}_Q$ for every $Q\subsetneq \Gamma$, then
\begin{equation}\label{1102-1}
A_\Gamma \leq \min \{A_Q+\sum_{\substack{i\in I\backslash Q}} A_i:Q\subsetneq \Gamma\},
\end{equation}
and
\begin{equation}
A_\Gamma \leq \sum_{i\in \Gamma }A_i.
\end{equation}
Moreover,
\begin{equation}\label{tab1122-2}
A_{ij}>\max\{A_i, A_j\}, ~\text{ for any } i\neq j.
\end{equation}
where we set $A_\Gamma=A_{ij}$ when $\Gamma=\{i,j\}$, $A_{\Gamma}$ is defined in \eqref{tem529-10} and $A_i$ is defined in \eqref{tem529-1}.
\end{lemma}

\begin{proof}[\bf Proof]
Firstly, we show that \eqref{1102-1} is true. Note that $|\Gamma|=2$ or $3$, without loss of generality, we only show that $A\leq A_{12}+A_3$.

{\bf Step 1.}  Assume that $I_{12}(\varphi^{12}_1,\varphi^{12}_2)=A_{12}$. Consider the function $h(t_1,t_2,t_3):=I(t_1\varphi^{12}_1,t_2\varphi^{12}_2,t_3\omega_3)$. Note that $\beta_{ij}>0$ for any $i\neq j$, then we see that $h(t_1,t_2,t_3) \leq g_{12}(t_1,t_2)+g_3(t_3)$,
where
\begin{equation}
g_{12}(t_1,t_2):=\frac{1}{2}\sum_{i=1}^2t_i^2\|\varphi^{12}_i\|_i^2-\frac{1}{6}\sum_{i=1}^2t_i^{6}
\int_{\mathbb{R}^3}|\varphi^{12}_i|^{6}-\frac{1}{3}t_1^3t_2^3\int_{\mathbb{R}^3}\beta_{12}|\varphi^{12}_1|^3|\omega^{12}_2|^3,
\end{equation}
\begin{equation}
g_3(t_3):=\frac{1}{2}t_3^2\|\omega_3\|_3^2-\frac{1}{6}t_3^{6}\int_{\mathbb{R}^3}|\omega_3|^{6}.
\end{equation}
It is easy to see that $\max_{t_{3}>0}g_{3}(t_{3})=A_3$.

{\bf Step 2.}  We claim that
\begin{equation}\label{tab1122-1-4}
\max_{t_{1},t_{2}\geq0}g_{12}(t_{1},t_{2})=g_{12}(1,1)=A_{12}.
\end{equation}

{ \bf Firstly}, note that $\beta_{ij}\geq0$ for any $i\neq j$, then we see that
\begin{equation}\label{tem103-2}
g_{12}(t_1,t_2)\leq \frac{1}{2}\sum_{i=1}^{2} t_i^2\left\|\varphi^{12}_i \right\|_i^2-\sum_{i=1}^{2} t_i^6\int_{\mathbb{R}^3}|\varphi^{12}_i|^{6} \to -\infty ,\quad \text{ as } |\mathbf{t}|\to +\infty,
\end{equation}
where $\mathbf{t}:=(t_1,t_2)$. Thus, by \eqref{tem103-2} we see that $g_{12}(t_1,t_2)$ has a global maximum point in $\overline{(\R^+)^2}$. Assume the global maximum point $\mathbf{t}=(s_1,s_2)$ belongs to  $\partial \overline{(\R^+)^2}$. Without loss of generality, we assume that $s_1=0$ and $s_2>0$, then
$$
g_{12}(0,s_2)=\frac{1}{2} s_2^2\left\|\varphi^{12}_2 \right\|_2^2-\frac{1}{6}s_2^6\int_{\mathbb{R}^3}\beta_{22} |\varphi^{12}_2|^6.
$$
For  $s>0$ small enough, we have
$$
g_{12}(s,s_2)-g_{12}(0,s_2)=\frac{1}{2}s^2\left\|\omega_1 \right\|_1^2-\frac{1}{6}s^6\int_{\mathbb{R}^3}\beta_{11}|\omega_1|^6-\frac{1}{3}s^3 s_2^3\int_{\mathbb{R}^3}\beta_{12} |\omega_1|^3|\varphi^{12}_2|^3>0,
$$
which contradicts to the fact that $(0,s_2)$ is a global maximum of $\overline{(\R^+)^2}$. Thus,  the global maximum point $(s_1,s_2)$ is a interior point in $ (\R^+)^2$, and so $(s_1,s_2)$ is an critical point.

Based on the above arguments, we also know that
\begin{equation*}
\max_{t_{1},t_{2}\geq0}g_{12}(t_{1},t_{2})>\max\left\{\max_{t_{1}>0}g_{12}(t_{1},0), \max_{t_{2}>0}f_{12}(0,t_{2})\right\}=\max\{A_1, A_2\}.
\end{equation*}

{\bf Secondly}, we prove that the critical point of $g_{12}(t_{1},t_{2})$ is unique in $(\R^+)^2$. Consider
\begin{equation}
\tilde{g}_{12}(t_1,t_2)=\frac{1}{2}\sum_{i=1}^{2} t_i^{\frac{2}{3}}\left\|\varphi^{12}_i \right\|_i^2-\frac{1}{6}\sum_{i,j=1}^{2} t_it_j\int_{\mathbb{R}^3}\beta_{ij} |\varphi^{12}_i|^3|\varphi^{12}_j|^3 .
\end{equation}
By a direct calculation,
\begin{equation}
\begin{aligned}
&\frac{\partial \tilde{g}_{12}}{\partial t_i}(t_1,t_2)=\frac{1}{3}t_i^{-\frac{1}{3}}\left\|\varphi^{12}_i \right\|_i^2-\frac{1}{3}\sum_{j=1}^2 t_j\int_{\mathbb{R}^3}\beta_{ij} |\varphi^{12}_i|^3|\varphi^{12}_j|^3 ,\quad 1\leq i \leq 2,\\
			&\frac{\partial^2 \tilde{g}_{12}}{\partial t_i^2}(t_1,t_2)=-\frac{1}{9}t_i^{-\frac{4}{3}}\left\|\varphi^{12}_i \right\|_i^2-\frac{1}{3}\int_{\mathbb{R}^3}\beta_{ii} |\varphi^{12}_i|^6 , \quad 1\leq i \leq 2,\\
&\frac{\partial^2 \tilde{g}_{12}}{\partial t_i \partial t_j}(t_1,t_2)=-\frac{1}{3}\int_{\mathbb{R}^3}\beta_{12} |\varphi^{12}_1|^3|\varphi^{12}_2|^3.
\end{aligned}
\end{equation}
Thus the Hessian matrix of $\tilde{g}_{12}$ is
\begin{equation}
	\begin{aligned}
		H(\tilde{g}_{12})& =-\frac{1}{9}\left(\begin{matrix}
				t_1^{-\frac{4}{3}}\|\varphi^{12}_1 \|_1^2 & 0 \\
	0 & t_2^{-\frac{4}{3}}\|\varphi^{12}_2 \|_2^2
		\end{matrix}\right)
-\frac{1}{3}\left(\begin{matrix}
			\int_{\mathbb{R}^3}\beta_{11} |\varphi^{12}_1|^6 & \int_{\mathbb{R}^3}\beta_{12} |\varphi^{12}_1|^3|\varphi^{12}_2|^3\\
			\int_{\mathbb{R}^3}\beta_{12} |\varphi^{12}_1|^3|\varphi^{12}_2|^3  & \int_{\mathbb{R}^3}\beta_{22} |\varphi^{12}_2|^6
		\end{matrix}\right)\\
	&=: -\frac{1}{9}B(\mathbf{t})-\frac{1}{3}M_B[\varphi^{12}_1,\varphi^{12}_2],
	\end{aligned}	
\end{equation}
where $ \mathbf{t}=(t_1,t_2)\in (\R^+)^2 $. Note that the matrix $B(\mathbf{t})$ is positive definite. Next, we claim that the matrix
$M_B[\varphi^{12}_1,\varphi^{12}_2]$ is also positive definite. In fact, since $\beta_{ij}<1$ for any $i\neq j$, then we have
\begin{align*}
  \int_{\mathbb{R}^3}|\varphi^{12}_1|^6 \int_{\mathbb{R}^3}|\varphi^{12}_2|^6-\left(\beta_{12}\int_{\mathbb{R}^3}|\varphi^{12}_1|^3|\varphi^{12}_2|^3\right)^2 >0,
\end{align*}
which yields that the matrix $M_B[\varphi^{12}_1,\varphi^{12}_2]$ is positive definite.
 Therefore the Hessian matrix of $\tilde{g}_{12}$ is negative definite, which implies that $\tilde{g}_{12}$ has a unique critical point. Therefore, the critical point must be the global maximum point. Notice that $ \tilde{g}_{12}(t_1^3,t_2^3) = g_{12}(t_1,t_2)$,  thus $g_{12}$ has a unique critical point, which must be the global maximum point. Since $(1,1)$ is a critical point of $g_{12}(t_{1},t_{2})$, then $(1,1)$ is a maximum point of $g_{12}$. Therefore,
\begin{equation}
\max_{t_{1},t_{2}>0}g_{12}(t_{1},t_{2})=g_{12}(1,1)=A_{12}.
\end{equation}
Hence, based on the above arguments we get that
\begin{equation}\label{tem322-2}
\max_{t_{1},t_{2},t_3>0}h(t_{1},t_{2},t_3)\leq A_{12}+A_3.
\end{equation}

{\bf Step 3.} We show that there exist $\widetilde{t}_1,\widetilde{t}_2,\widetilde{t}_3>0$ such that $(\widetilde{t}_1\varphi^{12}_1,\widetilde{t}_2\varphi^{12}_2,
\widetilde{t}_3\omega_3)\in\mathcal{M}$. It is easy to see that  $h(t_1,t_2,t_3)$ has a global maximum in $\overline{(\R^+)^3}$. Following the proof of
\cite[Proposition 2.1]{LiuYouZou2023}, any global maximum point of
$h(t_1,t_2,t_3) $ does not belong to $\partial \overline{(\R^+)^3}$, and any global maximum point $(\widetilde{t}_1,\widetilde{t}_2,\widetilde{t}_3)$ is an interior point in $(\R^+)^3$. It is standard to see that
\begin{equation}\label{tem322-1}
(\widetilde{t}_1\varphi^{12}_1,\widetilde{t}_2\varphi^{12}_2,
\widetilde{t}_3\omega_3)\in\mathcal{M} \text{ and } A\leq I(\widetilde{t}_1\varphi^{12}_1,\widetilde{t}_2\varphi^{12}_2,
\widetilde{t}_3\omega_3)\leq A_{12}+A_3.
\end{equation}
If $|Q|=1$, then by \eqref{1102-1} we get that
$$
A_\Gamma \leq \sum_{i\in \Gamma }A_i.
$$
\end{proof}


Set
\begin{equation}\label{tem529-6}
\widehat{\beta}_2:=\min\left\{\frac{C_3S^{\frac{3}{2}}}{2C_4^{\frac{1}{2}}(6\overline{C})^{\frac{3}{2}}},
\frac{C_3S^3}{12(6\overline{C})^2}\right\}
\end{equation}
where $C_3, C_4$ are defined in Lemma \ref{Estimate2-3}, $S$ is defined in \eqref{Constant-1125} and $\overline{C}$ is defined in \eqref{Constant-2-1}.

\begin{lemma}\label{estimation4-4}
Suppose that $\beta_{ij}\in (0, \widehat{\beta}_2)$ and let $\mathbf{v}_{n}\in \mathcal{M}_\Gamma$ be a minimizing sequence of $A_\Gamma$ with $I(\mathbf{v}_{n})\leq 2\overline{C}$, and $\mathbf{v}_{n}\rightharpoonup \mathbf{0}$. Then for any $\widetilde{\sigma}>0$ and for every $\widetilde{\delta}\in (-\widetilde{\sigma},0)\cup (0,\widetilde{\sigma})$, there exists $\widetilde{\rho}\in (-\widetilde{\delta},0)\cup (0,\widetilde{\delta})$ such that, up to a subsequence,
\begin{equation}\label{estimation4-1}
\text{ either } ~\sum_{i\in \Gamma}\int_{B_{\widetilde{\sigma}+\widetilde{\rho}}}|\nabla v_i^n|^2\rightarrow 0 ~\text{ or }
~\sum_{i\in \Gamma}\int_{\mathbb{R}^3 \backslash B_{\widetilde{\sigma}+\widetilde{\rho}}}|\nabla v_i^n|^2\rightarrow 0.
\end{equation}
\end{lemma}

\begin{proof}[\bf Proof]
Without loss of generality, we present the proof for the case of $\Gamma=\{1, 2, 3\}$. The idea of the following proof is similar to that of Lemma \ref{estimation3}. By Lemma \ref{PS-sequence-3}, $I'(v_1^n,v_2^n,v_3^n)\rightarrow 0$ as $n\rightarrow \infty$. As in Lemma \ref{estimation3}, for any $i=1, 2, 3$ we consider
\begin{equation}
v_{i,1}^n(x)=
\begin{cases}
v_i^n(x)  ~~ \text{ if } x\in B_{\widetilde{\sigma}+\widetilde{\rho}},\\
\widetilde{\omega}_{i,1}^n(x)  ~~ \text{ if } x\in B_{\widetilde{\sigma}+\widetilde{\delta}}\setminus B_{\widetilde{\sigma}+\widetilde{\rho}},\\
0 ~~\text{ elsewhere},
\end{cases}
~~v_{i,2}^n(x)=
\begin{cases}
0  ~~ \text{ if } x\in B_{\widetilde{\sigma}-\widetilde{\delta}},\\
\widetilde{\omega}_{i,2}^n(x)  ~~ \text{ if } x\in B_{\widetilde{\sigma}+\widetilde{\rho}}\setminus B_{\widetilde{\sigma}-\widetilde{\delta}},\\
v_i^n(x) ~~\text{ elsewhere},
\end{cases}
\end{equation}
where $\widetilde{\omega}^n_{i,1}\rightarrow 0$ strongly in $H^1(B_{\widetilde{\sigma}+\widetilde{\delta}}\backslash B_{\widetilde{\sigma}+\widetilde{\rho}})$
and $\widetilde{\omega}^n_{i,2}\rightarrow 0$ strongly in $H^1(B_{\widetilde{\sigma}+\widetilde{\rho}}\backslash B_{\widetilde{\sigma}-\widetilde{\delta}})$.
Obviously,
\begin{equation}\label{1125-1}
\|v_i^n\|_i^2= \|v_{i,1}^n\|_i^2+\|v_{i,2}^n\|_i^2+o(1).
\end{equation}

Next, we claim that
\begin{equation}\label{26-3-4}
\text{ either } \lim_{n\rightarrow\infty}\sum_{i=1}^3\|v_{i,1}^n\|^2=0 \text{ or }
\lim_{n\rightarrow\infty}\sum_{i=1}^3\|v_{i,2}^n\|^2=0.
\end{equation}
By contradiction, if \eqref{26-3-4} does not hold, then up to a subsequence,
\begin{equation}\label{26-4-4}
\text{ both } \sum_{i=1}^3\lim_{n\rightarrow\infty}\|v_{i,1}^n\|^2>0 \text{ and }
\sum_{i=1}^3\lim_{n\rightarrow\infty}\|v_{i,2}^n\|^2>0.
\end{equation}
Set
\begin{equation}
b_i:=\lim_{n\rightarrow\infty}\|v_{i,1}^n\|^2, \quad \mathbf{b}:=(b_1, b_2, b_3).
\end{equation}
It follows from \eqref{26-4-4} that $\sum_{i=1}^3b_i>0$, which implies that there exists at least one component of $\mathbf{b}$ is positive.

Assume that any component of $\mathbf{b}$ is positive, that is, $b_i=\lim_{n\rightarrow \infty}\|v_{i,1}^n\|^2>0$ for any $i=1,2,3$. Note that the norm $\|\cdot\|_{i}$ is equivalent to $\|\cdot\|$ for any $i=1,2,3$ and
\begin{equation}\label{26-5-4}
\|v_{i,1}^n\|_i^2=\sum_{j=1}^3\beta_{ij}\int_{\R^3}|v_{i,1}^n|^{3}|v_{j,1}^n|^{3}+o(1), \quad i=1,2,3,
\end{equation}
which yields that
\begin{equation}\label{2325-4}
\liminf_{n\rightarrow \infty}|v_{i,1}^n|_{6}^{6}>0.
\end{equation}
We claim that there exist $t_i^n>0$,$i=1,2,3$ such that
\begin{equation}
 \left(t_1^nv_{1,1}^n,t_2^nv_{2,1}^n,t_3^nv_{3,1}^n\right)\in \mathcal{M} \text{ and }  \lim_{n\rightarrow \infty}t_i^n=1.
\end{equation}
In order to prove this claim, we consider the following function
\begin{align*}
  f(t_1,t_2,t_3) & := I(t_1v_{1,1}^n,t_2v_{2,1}^n,t_3v_{3,1}^n) \\
   & = \frac{1}{2}\sum_{i=1}^{3}t_i^2\|v_{i,1}^n\|_i^2-\frac{1}{6}
   \sum_{i,j=1}^{3}t_i^3t_j^3\beta_{ij}\int_{\mathbb{R}^3}|v_{i,1}^n|^3|v_{j,1}^n|^3.
\end{align*}
By \eqref{1125-1} we get that $\|v_{i,1}^n\|_i^2\leq 7\overline{C}$ for $n$ large enough. Combining this with Lemma \ref{Estimate2-3} and $\beta_{ij}\geq 0$ for any $i\neq j$ we have
\begin{align}
   f(t_1,t_2,t_3) & \leq \frac{1}{2}\sum_{i=1}^{3}t_i^2\|v_{i,1}^n\|_i^2-\frac{1}{6}
   \sum_{i=1}^{3}t_i^6\int_{\mathbb{R}^3}|v_{i,1}^n|^6 \\
   & \leq \sum_{i=1}^{3}\left(\frac{7\overline{C}}{2}t_i^2-\frac{1}{6}C_3
   t_i^6\right),
\end{align}
where $C_3$ is defined in Lemma \ref{Estimate2-3}.
Hence,
\begin{equation}\label{1125-3}
f(t_1,t_2,t_3)<0 \text{ when } t_i\in (\overline{t}, +\infty), \text{ where } \overline{t}= \sqrt[4]{\frac{7\overline{C}}{C_3}}.
\end{equation}
Therefore, $f(t_1,t_2,t_3)$ has a global maximum in $\overline{(\R^+)^3}$. Similarly to the proof of Lemma \ref{Energyestimates6}, the global maximum point $(t_1^n,t_2^n,t_3^n)$ is an interior point in $(\R^+)^3$. It follows that
\begin{equation*}
(t_1^nv_{1,1}^n,t_2^nv_{2,1}^n,t_3^nv_{3,1}^n)\in \mathcal{M} \text{ and } 0<t_i^n<\overline{t},
\end{equation*}
where $\overline{t}$ is defined in \eqref{1125-3}. Denote $\widehat{v}_{i}^n=t_i^nv_{i,1}^n, i=1,2,3$, then $(\widehat{v}_{1}^n, \widehat{v}_{2}^n, \widehat{v}_{3}^n)\in \mathcal{M}$ and $I(\widehat{v}_{1}^n, \widehat{v}_{2}^n, \widehat{v}_{3}^n)\leq 2\overline{t}^2\overline{C}$. Following the proof of Lemma \ref{Estimate2-3}, there exists a constant $\widetilde{C}_3>0$ dependent on $\lambda_i, S, \overline{C}$ such that
\begin{equation}\label{23125-2}
|\widehat{v}_{i}^n|_6^6\geq \widetilde{C}_3 ~\text{ if } ~0<\beta_{ij}<\widehat{\beta}_2.
\end{equation}
Set
\begin{equation*}
B_{ij}^n:= \int_{\mathbb{R}^3}\beta_{ij}|v_{i,1}^n|^3|v_{j,1}^n|^3.
\end{equation*}
Note that $(t_1^nv_{1,1}^n,t_2^nv_{2,1}^n,t_3^nv_{3,1}^n)\in \mathcal{M}$, then we have
\begin{equation}\label{2325-9}
|t_i^n|^2\|v_{i,1}^n\|_i^2=\sum_{j=1}^3|t_i^n|^{3}|t_j^n|^{3}B_{ij}^n, ~i=1,2,3.
\end{equation}
Up to a subsequence, we can assume that
\begin{equation*}
\overline{B}_{ij}:=\lim_{n\rightarrow\infty} B_{ij}^n= \lim_{n\rightarrow\infty}\int_{\mathbb{R}^3}\beta_{ij}|v_{i,1}^n|^3|v_{j,1}^n|^3,\quad
a_i:= \lim_{n\rightarrow\infty} \|v_{i,1}^n\|_i^2,\quad
\overline{t}_i:=\lim_{n\rightarrow\infty} t_i^n.
\end{equation*}
It follows from \eqref{23125-2} that
\begin{equation}\label{2325-6}
\lim_{n\rightarrow\infty}|\widehat{v}_{i}^n|_6^6= \overline{t}_i^6 \lim_{n\rightarrow\infty}|v_{i,1}^n|_6^6\geq \widetilde{C}_3>0,
\end{equation}
which yields that
\begin{equation}\label{2325-6}
\overline{t}_i> 0 \text{ for any } i=1,2,3.
\end{equation}
Next, we prove that $\overline{t}_i=1, i=1,2,3.$
Consider
\begin{equation*}
\overline{f}(s_1,s_2,s_3):=\frac{1}{2}\sum_{i=1}^3a_is_i^2-\frac{1}{6}\sum_{i,j=1}^3\overline{B}_{ij}s_i^3s_j^3.
\end{equation*}
By \eqref{26-5-4} and \eqref{2325-9} we know that both $(1,1,1)$ and $(\overline{t}_1, \overline{t}_2, \overline{t}_3)$ are critical points of $\overline{f}(s_1,s_2,s_3)$.
We claim that the critical point of $\overline{f}(s_1,s_2,s_3)$ is unique in $(\R^+)^3$. In fact, for any $i=1,2,3$ we have
\begin{align*}
 \int_{\mathbb{R}^3}|v_{i,1}^n|^6- \sum_{j\neq i}\left|\int_{\mathbb{R}^3}\beta_{ij}|v_{i,1}^n|^3|v_{j,1}^n|^3\right|
   & \geq C_3-\widehat{\beta}_2 S^{-\frac{3}{2}} |v_{i,1}^n|_6^3\|v_{j,1}^n\|_j^3\\
   & \geq C_3-\widehat{\beta}_2 S^{-\frac{3}{2}}C_4^{\frac{1}{2}}(6\overline{C})^{\frac{3}{2}}\\
   & \geq \frac{1}{2}C_3,
\end{align*}
where $C_3,C_4$ are defined in Lemma \ref{Estimate2-3}. Thus, for any $i=1,2,3$ we have
\begin{equation*}
\overline{B}_{ii}-\sum_{j\neq i}\overline{B}_{ij}\geq \frac{1}{2}C_3.
\end{equation*}
By Gershgorin circle theorem we know that the matrix $\overline{B}=(\overline{B}_{ij})_{1\leq i,j\leq3}$ is positive definite. Then following the arguments of Lemma \ref{Energyestimates6}, we can get that the critical point of $\overline{f}(s_1,s_2,s_3)$ is unique in $(\R^+)^3$. By \eqref{2325-6} we get that $\overline{t}_i>0, i=1,2,3$, thus $(\overline{t}_1, \overline{t}_2, \overline{t}_3)=(1,1,1)$.
Therefore,
\begin{align*}
A & = \lim_{n\rightarrow \infty} I(\mathbf{v}_n)= \frac{1}{3}\lim_{n\rightarrow \infty}\sum_{j=i}^3|t_i^n|^2\|v_{i,1}^n\|_i^2+\frac{1}{3}\lim_{n\rightarrow \infty}\sum_{j=i}^3\|v_{i,2}^n\|_i^2\\
   & > \frac{1}{3}\lim_{n\rightarrow \infty}\sum_{j=i}^3|t_i^n|^2\|v_{i,1}^n\|_i^2\geq A,
\end{align*}
which contradicts with Lemma \ref{Energyestimates6}. Therefore, at least one component of $\mathbf{b}$ is negative. Similarly as above, by virtue of Lemma \ref{Energyestimates6} and \ref{estimation3}, ones can show that any component of $\mathbf{b}$ should be negative. This yields that \eqref{26-4-4} does not hold and then \eqref{estimation4-1} and \eqref{26-3-4} are true.

\end{proof}

\subsubsection{\bf Proof of Proposition \ref{1103-1}}

Recall that $ A_i<\frac{1}{3}\mathcal{S}^{\frac{3}{2}}$ (see \eqref{tem529-1}) for every $1 \leq i \leq 3.$ Set				
\begin{equation} \label{3.6}
\pi:=\frac{1}{2}\min_{1\leq i \leq 3}\{\mathcal{S}^3- (3A_i)^2\}>0,
\end{equation}
then we have
\begin{equation} \label{3.7}
(3A_i)^2 <\mathcal{S}^{3}-\pi, \quad 1\leq i \leq 3.
\end{equation}
Set
\begin{equation}\label{tem529-4}
\widehat{\beta}_3:=\frac{\pi S^3}{2\mathcal{S}^3(6\overline{C})^2} \quad
\text{ and }
\widehat{\beta}:=\min\{\widehat{\beta}_1, \widehat{\beta}_2, \widehat{\beta}_3, 1\},
\end{equation}
where $\widehat{\beta}_1$ is defined in \eqref{tem529-5}, $\widehat{\beta}_2$ is defined in \eqref{tem529-6}. From now on, we assume that $\beta_{ij}\in (0, \widehat{\beta})$ and fix such an $\beta_{ij}$.		

\begin{proof}[\bf The proof of Proposition \ref{1103-1}]
We will proceed by mathematical induction on the number of the equations in the subsystem. Set $|\Gamma|=q$, that is $q$ the number of the equations in the subsystem, and $q=1,\ldots, 3$.
		
When $q=1$, the system reduces to the following problem
\begin{equation}
-\Delta v-\frac{\lambda_{i}}{|x|^2}v=v^{2^*-1}, ~ x\in \R^N,
\end{equation}
and we know Proposition \ref{1103-1} is true.

Suppose by induction hypothesis that Proposition \ref{1103-1} holds true for every level $\mathcal{C}_\Gamma$ with $|\Gamma|\leq q$  for some $1\leq q\leq 2$, we need prove Proposition \ref{1103-1} for $\mathcal{C}_\Gamma$ with $|\Gamma|=q+1$. Without loss of generality, we only present the proof for $\Gamma=\{1,\ldots,q+1\}$.  By induction hypothesis, Lemma \ref{Energyestimates6} is true for $\mathcal{C}_\Gamma$.

Let $\widetilde{\mathbf{u}}_\Gamma^n=(\widetilde{u}_1^n, \cdots,\widetilde{u}_{q+1}^n)$ be a minimizing sequence for the level $A_\Gamma$. Recall that
\begin{equation}
G(\widetilde{\mathbf{u}}_\Gamma^n)=\sum_{i \in \Gamma}\left(|\nabla u_i|^2-\frac{\lambda_i}{|x|^2}\right).
\end{equation}
Similarly to the proof of Theorem \ref{thm2}-(1), we see that there exists a sequence $\{(u_1^n,\cdots,u_{q+1}^n)\}\in \mathcal{M}_\Gamma$ such that
\begin{equation}
\lim_{n\rightarrow \infty}I(\mathbf{u}^n)=A_\Gamma  \text{ and } \lim_{n\rightarrow \infty}I'(\mathbf{u}^n)=0.
\end{equation}

Observe that $\{(u_1^n,\cdots,u_{q+1}^n)\}$ are uniformly bounded in $\mathbf{D}$, then up to a subsequence we suppose that
$(u_1^n,\cdots,u_{q+1}^n)\rightharpoonup (u_1,\cdots,u_{q+1})$ weakly in $\mathbf{D}$. Hence, $I'(u_1,\cdots,u_{q+1})=0$. Next, we will present that $u_i\neq 0$ for any $i=1,\cdots,q+1$ and $\mathbf{u}=(u_1,\cdots,u_{q+1})\in \mathcal{M}_\Gamma$ by contradiction.

{\bf Case 1}. $(u_1,\cdots,u_{q+1})=(0,\cdots,0)$.\\
Based upon Lemma \ref{estimation4-4}, following the arguments in the proof of \cite[Theorem 1.2]{Zou2015}, we can get that
\begin{equation}
\int_{\R^3}\frac{|u_i^n|^2}{|x|^2}=o(1), ~\text{ for any } i\in \Gamma.
\end{equation}
Set
\begin{equation}
d_i:= \lim_{n\rightarrow\infty} \int_{\R^3}|\nabla u_i^n|^2.
\end{equation}
Combining this with $(u_1^n,\cdots,u_{q+1}^n)\in \mathcal{M}_\Gamma$, we have
\begin{align}
 \int_{\R^3}|\nabla u_i^n|^2 & =\int_{\R^3}|u_i^n|^6+ \sum_{j\neq i,j\in \Gamma}\int_{\R^3}\beta_{ij}|u_i^n|^3|u_j^n|^3+o(1)  \nonumber\\
   & \leq \mathcal{S}^{-2}\left(\int_{\R^3}|\nabla u_i^n|^2\right)^3+ \beta S^{-3}(6\overline{C})^{\frac{3}{2}}\left(\int_{\R^3}|\nabla u_i^n|^2+o(1)\right)^{\frac{3}{2}},
\end{align}
which yields that
\begin{equation}
d_i\leq \mathcal{S}^{-3}d_i^3+\beta S^{-3}(6\overline{C})^{\frac{3}{2}}d_i^{\frac{3}{2}}\leq \mathcal{S}^{-3}d_i^3+\beta S^{-3}(6\overline{C})^2d_i, ~\text{ for any } i\in \Gamma.
\end{equation}
Thus, we get that
\begin{equation}
d_i^2\geq \mathcal{S}^{3}-\beta \mathcal{S}^{3}S^{-3}(6\overline{C})^2\geq \mathcal{S}^{3}- \pi> (3A_i)^2,
\end{equation}
and so
\begin{equation*}
d_i > 3A_i.
\end{equation*}
Therefore,
\begin{align*}
  A_\Gamma & =\lim_{n\rightarrow \infty}\frac{1}{3}\sum_{i\in \Gamma}\|u_i^n\|_i^2
   = \lim_{n\rightarrow \infty}\frac{1}{3}\sum_{i\in \Gamma}\int_{\R^3}|\nabla u_i^n|^2\\
   &= \frac{1}{3}\sum_{i\in \Gamma} d_i
   > \sum_{i\in \Gamma}A_i,
\end{align*}
which contradicts with Lemma \ref{Energyestimates6}.

{\bf Case 2}. Only one component of $\mathbf{u}$ is nontrivial.\\
Without loss of generality, we may assume that $u_1\neq 0, u_2=0=\cdots=u_{q+1}=0$. Since $I'(u_1,\cdots,u_{q+1})(u_1,\cdots,0)=0$, then
\begin{equation*}
\|u_1\|_1^2=\int_{\R^3}|u_1|^6\leq S_1^{-3}\|u_1\|_1^6,
\end{equation*}
and so $\|u_1\|_1^2\geq S_1^{\frac{3}{2}}.$

{\bf Case 2.1}. $u_1^n\rightarrow u_1$ strongly in $D^{1,2}(\R^3)$.\\
Note that $u_1^n\rightarrow u_1$ strongly in $L^6(\R^3)$ and $u_i^n\rightharpoonup 0$ weakly in $D^{1,2}(\R^3)$, $i=2,\cdots,q+1$, then we see that
\begin{equation}\label{1104-1}
\int_{\R^3}|u_1^n|^3|u_i^n|^3=o(1), ~\text{ for any } i\in \Gamma.
\end{equation}
Note that $q=1$ or $2$. Then by \eqref{1104-1} and $\mathbf{u}_\Gamma^n \in \mathcal{M}_\Gamma$ we have
\begin{equation}
\|u_{i}^n\|_i^2=\int_{\R^3}|u_{i}^n|^{6}+\beta_{ij}\int_{\R^3}|u_{i}^n|^{3}|u_{j}^n|^{3}+o(1), ~i\neq 1, i\in \Gamma.
\end{equation}
Similarly to the proof of Lemma \ref{estimation3} we know that there exist $\tau_2^n,\cdots,\tau_{q+1}^n$ such tat
\begin{equation}
(\tau_2^nu_{2}^n,\cdots, \tau_{q+1}^nu_{q+1}^n)\in \mathcal{M}_{\{2,\cdots, q+1\}} ~\text{ and } \lim_{n\rightarrow \infty}\tau_i^n=1.
\end{equation}
Therefore,
\begin{align}
A_\Gamma & \geq \frac{1}{3}\|u_1\|_1^2+\frac{1}{3}\lim_{n\rightarrow \infty}\sum_{i=2}^{q+1}\|u_{i}^n\|_i^2\nonumber\\
&\geq \frac{1}{3}S_1^{\frac{3}{2}}+\frac{1}{3}\lim_{n\rightarrow \infty}\sum_{i=2}^{q+1}\tau_i^n|\|u_{i}^n\|_i^2 \nonumber\\
&\geq A_{1}+A_{\{2,\cdots, q+1\}},
\end{align}

{\bf Case 2.2}. $\lim_{n\rightarrow \infty}\|u_1^n -u_1\|_1^2>0$.\\
Set $\omega_1^n=u_1^n -u_1$.
Similarly to the proof of Lemma \ref{estimation4-4} there exist $\alpha_i^n>0$, $i\in \Gamma$ such that
\begin{equation}
 \left(\alpha_1^n\omega_{1}^n,\cdots,\alpha_{q+1}^nu_{q+1}^n\right)\in \mathcal{M}_\Gamma \text{ and }  \lim_{n\rightarrow \infty}\alpha_i^n=1.
\end{equation}
Therefore,
\begin{align*}
A_\Gamma &= \frac{1}{3}\|u_1\|_1^2+\frac{1}{3}\lim_{n\rightarrow \infty}(\|\omega_{1}^n\|_1^2+\sum_{i=2}^{q+1}\|u_{i}^n\|_i^2)\\
   & >\frac{1}{3}\lim_{n\rightarrow \infty}(\alpha_1^n\|\omega_{1}^n\|_1^2+\sum_{i=2}^{q+1}\alpha_i^n\|u_{i}^n\|_i^2)\\
   & \geq A_\Gamma,
\end{align*}
which contradicts with Lemma \ref{Energyestimates6}.

{\bf Case 3}. Only two component of $\mathbf{u}$ are nontrivial. Without loss of generality, we may assume that $u_1\neq 0, u_2\neq0, u_{q+1}=u_3=0$. Similarly to the proof of Theorem \ref{thm2}-(1) we see that {\bf Case 3} is impossible.
\end{proof}

\begin{proof}[\bf The proof of Theorem \ref{thm2}-(2)]
Taking $\Gamma=\{1, 2, 3\}$, by Proposition \ref{1103-1} we see that Theorem \ref{thm2}-(2) is true.
\end{proof}

\section{Nonexistence of least energy positive solutions for the mixed case}
\renewcommand{\theequation}{4.\arabic{equation}}
In this section, we focus on the nonexistence of  least energy solutions for the mixed case (i.e. $\beta_{12}>0, \beta_{13}<0, \beta_{23}<0$), and prove Theorem \ref{thm3}-\ref{thm4}.

\subsection{The case $N=4$}
In this subsection, we assume that $0<\beta_{12}<\widetilde{\beta}, \beta_{13}<0, \beta_{23}<0$, and show the proof of Theorem \ref{thm3}-(1).
Define
\begin{equation}\label{tem127-4}
\widetilde{\mathcal{M}}_{12}:=\left\{u_1\neq 0, u_2\neq 0,  \int_{\mathbb{R}^4}\left(|\nabla u_i|^{2}-\frac{\lambda_{i}}{|x|^2} u_i^{2}\right)\leq\sum_{j=1}^2\int_{\mathbb{R}^4}\beta_{ij}|u_i|^{2}|u_j|^{2}, i=1,2\right\}.
\end{equation}
and
\begin{equation}\label{tem127-2}
\widetilde{A}_{12}:=\inf_{(u_1,u_2)\in \widetilde{\mathcal{M}}_{12}}I_{12}(u_1,u_2).
\end{equation}

\begin{lemma}\label{tem1010-1}
Assume that $\lambda_1=\lambda_2$ and $0<\beta_{12}<\widetilde{\beta}$. Then we have
\begin{equation*}
\widetilde{A}_{12}=A_{12},
\end{equation*}
where $A_{12}$ is defined in \eqref{tem529-10}.
\end{lemma}
\begin{proof}
Based on inequality \eqref{tem127-4}, employing the proof of \cite[Theorem 1.1]{Zou2012} we see that
there exist $k,l>0$ such that $A_{12}=(\sqrt{k}\omega_1, \sqrt{l}\omega_1)$, where $\omega_1$ is a least energy positive solution of \eqref{scalarequation}. Similarly, we also can get that $\widetilde{A}_{12}=(\sqrt{k}\omega_1, \sqrt{l}\omega_1)$. Therefore, $\widetilde{A}_{12}=A_{12}$.
\end{proof}
Consider
\begin{equation*}
\widetilde{\mathcal{M}}_i:=\left\{u\neq 0, ~\int_{\mathbb{R}^N}\left(|\nabla u|^{2}-\frac{\lambda_{i}}{|x|^2} u^{2}\right)\leq\int_{\mathbb{R}^N}|u|^{2^*}\right\},
\end{equation*}
and $\widetilde{A}_i:=\inf_{u\in \widetilde{\mathcal{M}}} \frac{1}{N}\int_{\mathbb{R}^N}\left(|\nabla u|^{2}-\frac{\lambda_{i}}{|x|^2} u^{2}\right)$.
\begin{lemma}\label{tem512-1}
Assume that $N\geq3$, then $\widetilde{A}_i=A_i$, where $A_i$ is defined in \eqref{tem529-10}.
\end{lemma}
\begin{proof}
 Note that $\widetilde{A}_i\leq A_i$. If $v\in \widetilde{\mathcal{M}}_i$ then take $t^{2^*-2}:=\int_{\mathbb{R}^N}\left(|\nabla v|^{2}-\frac{\lambda_{i}}{|x|^2} v^{2}\right)/(\int_{\mathbb{R}^N}|v|^{2^*})\leq 1 $, so that $tv\in \mathcal{M}_i$, where $\mathcal{M}_i$ is defined in \eqref{tem107-4}. Then
\[
A_i \leq  \frac{t^{\frac{2}{2^*-2}}}{N}\int_{\mathbb{R}^N}\left(|\nabla v|^{2}-\frac{\lambda_{i}}{|x|^2} v^{2}\right),
\]
and by taking the infimum for $v\in \widetilde{\mathcal{M}}$ we get $A_i\leq \widetilde{A}_i$ and so $\widetilde{A}_i=A_i$.
\end{proof}

\begin{lemma}\label{tem106-2}
Assume that $0<\beta_{12}<\widetilde{\beta}, \beta_{13}<0, \beta_{23}<0$, then $A\geq A_3+A_{12}$.
\end{lemma}
\begin{proof}
This was proved in \cite[Theorem 1]{Lin-Wei2005} for subcritical case, but we sketch it here for completeness. Since $\beta_{13}<0, \beta_{23}<0$, then for any $\mathbf{u}=(u_1,u_2,u_3)\in \mathcal{M}$ we have,
$
I(u_1,u_2,u_3)\geq I_{12}(u_1,u_2)+I_3(u_3),
$
where
\begin{equation*}
I_{12}(u_1,u_2)=\sum_{i=1}^2\frac{1}{2}\int_{\mathbb{R}^4}\left(|\nabla u_i|^{2}-\frac{\lambda_{i}}{|x|^2} u_i^{2}\right)-\frac{1}{4}\sum_{i,j=1}^2\int_{\mathbb{R}^4}\beta_{ij}|u_i|^{2}|u_j|^{2},
\end{equation*}
and
\begin{equation*}
I_3(u_3)=\frac{1}{2}\int_{\mathbb{R}^4}\left(|\nabla u_3|^{2}-\frac{\lambda_{3}}{|x|^2} u_3^{2}\right)-\frac{1}{4}\int_{\mathbb{R}^4}|u_3|^{4}.
\end{equation*}
Thus, Lemma \ref{tem1010-1} and Lemma \ref{tem512-1} yields that
\begin{equation}\label{tem107-8}
I(u_1,u_2,u_3)\geq I_{12}(u_1,u_2)+I_3(u_3)\geq A_3+A_{12}.
\end{equation}
Therefore, $A \geq A_3+A_{12}$.
\end{proof}

\begin{lemma}\label{tem106-3}
Assume that $0<\beta_{12}<\widetilde{\beta}, \beta_{13}<0, \beta_{23}<0$, then $A\leq A_3+A_{12}$.
\end{lemma}
\begin{proof}
Note that $0<\beta_{12}<\widetilde{\beta}$, then $A_{12}$ is attained and we assume that $A_{12}=I_{12}(\omega^{12}_1,\omega^{12}_2)$. Recall that $z^3_\mu$ is defined in \eqref{tem512-3}. Since $z^3_\mu\rightharpoonup 0$ in $D^{1,2}(\mathbb{R}^4)$ as $\mu\rightarrow \infty$, then $(z^3_\mu)^2\rightharpoonup 0$ in $L^{2}(\mathbb{R}^4)$ as $\mu\rightarrow \infty$. Thus
\begin{equation}\label{tem107-12}
\int_{\mathbb{R}^4}|z^3_\mu|^{2}|\omega^{12}_i|^{2}\rightarrow 0 \text{ as } \mu\rightarrow \infty, ~\text{ for any } i=1,2.
\end{equation}
Set $v_1:= \omega^{12}_1, v_2:= \omega^{12}_2, v_3:= z^3_\mu$.
Then we have
\begin{equation}\label{tem107-17}
\int_{\mathbb{R}^4}\left(|\nabla v_3|^{2}-\frac{\lambda_{3}}{|x|^2} v_3^{2}\right)=\int_{\mathbb{R}^4}|v_3|^{4}=4A_3.
\end{equation}
Consider the following system
\begin{equation}\label{tem107-13}
\int_{\mathbb{R}^4}\left(|\nabla v_i|^{2}-\frac{\lambda_{i}}{|x|^2} v_i^{2}\right)=\sum_{j=1}^3t_j\int_{\mathbb{R}^4}\beta_{ij}|v_i|^{2}|v_j|^{2}, \quad i=1,2,3.
\end{equation}
Note that $0<\beta_{12}<\widetilde{\beta}\leq 1$, then the matrix $M_B[v_1, v_2]$ is positively definite, and so $det(M_B[v_1, v_2])>0$, which is independent on $\mu$. Define
\begin{equation*}
M_B^\infty:=
\left(
\begin{matrix}
M_B[v_1, v_2] & \mathbf{0} \\
\mathbf{0}  &  \int_{\mathbb{R}^4}|v_3|^{4}
\end{matrix}
\right).
\end{equation*}
Then $det(M_B^\infty)>0$, which is independent on $\mu$. Note that $det(M_B[v_1, v_2, v_3])$ is a continuous function about the variables $\int_{\mathbb{R}^4}\beta_{ij}|v_i|^{2}|v_j|^{2}$ with $1\leq i,j\leq 3$. It follows form \eqref{tem107-12} that
\begin{equation}\label{tem107-15}
det(M_B[v_1, v_2, v_3])>\frac{1}{2}det(M_B^\infty)>0 \text{ for large enough } \mu.
\end{equation}
Therefore, the system
\eqref{tem107-13} has a unique solution $(t_1^\mu, t_2^\mu, t_3^\mu)$ for large enough $\mu$. Note that every element of $M_B[v_1, v_2, v_3]$ is uniformly bounded, combining this with \eqref{tem107-15} and Cramer's rule we know that $t_i^\mu$ is uniformly bounded, $i=1,2,3$. Then we can assume that
\begin{equation*}
\widetilde{t}_i:=\lim_{\mu\rightarrow \infty}t_i^\mu.
\end{equation*}
Letting $\mu\rightarrow \infty$ in \eqref{tem107-13} we have
\begin{equation}\label{tem107-16}
\begin{cases}
\int_{\mathbb{R}^4}\left(|\nabla v_1|^{2}-\frac{\lambda_{1}}{|x|^2} v_1^{2}\right)=\sum_{j=1}^2\widetilde{t}_1\int_{\mathbb{R}^4}\beta_{12}|v_1|^{2}|v_2|^{2}\\
\int_{\mathbb{R}^4}\left(|\nabla v_2|^{2}-\frac{\lambda_{2}}{|x|^2} v_2^{2}\right)=\sum_{j=1}^2\widetilde{t}_2\int_{\mathbb{R}^4}\beta_{12}|v_1|^{2}|v_2|^{2}\\
\int_{\mathbb{R}^4}\left(|\nabla v_3|^{2}-\frac{\lambda_{3}}{|x|^2} v_3^{2}\right)=\widetilde{t}_3\int_{\mathbb{R}^4}|v_3|^{4}.
\end{cases}
\end{equation}
Note that $det(M_B^\infty)>0$, then the system \eqref{tem107-16} has a unique solution. It is easy to see that $(1,1,1)$ is a solution of system \eqref{tem107-16}. Thus, $\widetilde{t}_i=1, i=1,2,3$, that is,
\begin{equation}\label{tem107-20}
\lim_{\mu\rightarrow \infty}t_i^\mu=1, ~\text{ for any } i=1,2,3.
\end{equation}
We deduce from \eqref{tem107-13} that $(\sqrt{t_1^\mu}v_1, \sqrt{t_2^\mu}v_2, \sqrt{t_3^\mu}v_3)\in \mathcal{M}$. Therefore,
\begin{align}\label{tem107-18}
 A  & \leq I\left(\sqrt{t_1^\mu}v_1, \sqrt{t_2^\mu}v_2, \sqrt{t_3^\mu}v_3\right) \nonumber\\
   & =I_{12}(t_1^\mu v_1, t_2^\mu v_2)+\frac{1}{2}t_3^\mu\int_{\mathbb{R}^4}\left(|\nabla v_3|^{2}-\frac{\lambda_{3}}{|x|^2} v_3^{2}\right)-\frac{1}{4}(t_3^\mu)^2\int_{\mathbb{R}^4}|v_3|^{4} \nonumber\\
    & \quad -\frac{1}{2}\sum_{i=1}^2t_i^\mu t_3^\mu\int_{\mathbb{R}^4}\beta_{i3}|v_i|^{2}|v_3|^{2}.
\end{align}
Letting $\mu\rightarrow \infty$ in \eqref{tem107-18}, then by \eqref{tem107-20}, \eqref{tem107-12} and \eqref{tem107-17} we have
\begin{equation*}
A\leq A_3+A_{12}.
\end{equation*}
\end{proof}

\begin{proof}[\bf The proof of Theorem \ref{thm3}-(1)]
By contradiction, we assume that $A$ is achieved by a $(\overline{u}_1, \overline{u}_2, \overline{u}_3) \in \mathcal{M}$ with $\overline{u}_i\neq0, i=1,2,3$. It follows from Lemma \ref{tem106-2} and Lemma \ref{tem106-3} that $A=A_3+A_{12}$. Note that $\beta_{13}<0, \beta_{23}<0$, then similarly  to the proof of Lemma \ref{tem106-2} we have
\begin{align*}
   A & =I(\overline{u}_1, \overline{u}_2, \overline{u}_3) =I_{12}(\overline{u}_1, \overline{u}_2)+I_3(\overline{u}_3)-\frac{1}{2}\sum_{i=1}^2
\int_{\mathbb{R}^4}\beta_{i3}|\overline{u}_i|^{3}|\overline{u}_3|^{3}  \\
   & >I_{12}(\overline{u}_1, \overline{u}_2)+I_3(\overline{u}_3)\geq A_{12}+A_3,
\end{align*}
which is a contradiction.

\end{proof}

\subsection{The case for $N=3$}

In this subsection, we assume that $0<\beta_{12}<\widehat{\beta}, \beta_{13}<0, \beta_{23}<0$ and give the proof of Theorem \ref{thm3}-(2). The idea of the following proof is similar to that of Theorem \ref{thm3}-(1). But there are technical differences between $N=3$ and $N=4$. For example, when projecting on the Nehari manifold, the equations the coefficients comply to are a linear system for $N=4$ (see \eqref{tem107-16}) and the Cramer's rule can be used. While the corresponding equations is a nonlinear system for $N=3$ and the Cramer's rule is not available. Thus, we introduced a different technique to deal with the case $N=3$.

Define
\begin{equation}\label{tem127-5}
\widehat{\mathcal{M}}_{12}:=\left\{u_1\neq 0, u_2\neq 0,  \int_{\mathbb{R}^4}\left(|\nabla u_i|^{2}-\frac{\lambda_{i}}{|x|^2} u_i^{2}\right)\leq\sum_{j=1}^2\int_{\mathbb{R}^4}\beta_{ij}|u_i|^{3}|u_j|^{3}, i=1,2\right\}.
\end{equation}
and
\begin{equation}\label{tem127-5}
\widehat{A}_{12}:=\inf_{(u_1,u_2)\in \widehat{\mathcal{M}}_{12}}I_{12}(u_1,u_2).
\end{equation}

\begin{lemma}\label{tem1206-1}
Assume that $\lambda_1=\lambda_2$ and $0<\beta_{12}<\widehat{\beta}$. Then we have
\begin{equation*}
\widehat{A}_{12}=A_{12},
\end{equation*}
where $A_{12}$ is defined in \eqref{tem529-10}.
\end{lemma}
\begin{proof}
Based on inequality \eqref{tem127-4}, employing the proof of \cite[Theorme 1.1]{Guo-Zou2016} we know that
there exist $k_0,l_0>0$ such that $A_{12}=(\sqrt{k_0}\omega_1, \sqrt{l_0}\omega_1)$, where $\omega_1$ is a lease energy positive solution of \eqref{scalarequation}. Similarly, we also can get that $\widehat{A}_{12}=(\sqrt{k_0}\omega_1, \sqrt{l_0}\omega_1)$. Thus, $\widehat{A}_{12}=A_{12}$.
\end{proof}

\begin{lemma}\label{tem1012-25}
Assume that $0<\beta_{12}<\widehat{\beta}, \beta_{13}<0, \beta_{23}<0$, then $A\geq A_3+A_{12}$.
\end{lemma}
\begin{proof}
The proof is similar to that of Lemma \ref{tem106-2}, and so we omit it.
\end{proof}

\begin{prop}\label{tem1012-19}
Assume that $0<\beta_{12}<\widehat{\beta}, \beta_{13}<0, \beta_{23}<0$, then $A\leq A_3+A_{12}$.
\end{prop}

\medbreak

Next, we show the proof of Proposition \ref{tem1012-19}. Note that $0<\beta_{12}<\widehat{\beta}$, then by Proposition \ref{1103-1} we know that $A_{12}$ is attained and we assume that $A_{12}=I_{12}(\omega^{12}_1,\omega^{12}_2)$. Recall that $z^3_\mu$ is defined in \eqref{tem512-3}. Note that $z^3_\mu\rightharpoonup 0$ in $D^{12}(\mathbb{R}^3)$ as $\mu\rightarrow \infty$, then $(z^3_\mu)^3\rightharpoonup 0$ in $L^{2}(\mathbb{R}^3)$ as $\mu\rightarrow \infty$. Set $e_1:= \omega^{12}_1, e_2:= \omega^{12}_2, e_3:= z^3_\mu$. Then we have
\begin{equation}\label{tem1012-3}
\int_{\mathbb{R}^3}\left(|\nabla e_3|^{2}-\frac{\lambda_{3}}{|x|^2} e_3^{2}\right)=\int_{\mathbb{R}^3}|e_3|^{6}=3A_3,
\end{equation}
and
\begin{equation}\label{tem1012-2}
\int_{\mathbb{R}^3}|e_i|^{3}|e_3|^{3}\rightarrow 0 \text{ as } \mu\rightarrow \infty, ~\text{ for any } i=1,2.
\end{equation}

Similarly to Lemma \ref{Estimate2-3} we have
\begin{lemma}\label{tem1012-1}
Assume that $0<\beta_{12}<\widehat{\beta}$. Then there exists $C_3,C_4>0$ such that
\begin{equation*}
C_3 \leq \int_{\mathbb{R}^3}|e_1|^{6} \leq C_4, \quad C_3 \leq \int_{\mathbb{R}^3}|e_2|^{6}\leq C_4,
\end{equation*}
$C_3$ is independent on $\beta_{12}$.
\end{lemma}

Denote
\begin{equation}\label{tem1012-7}
\gamma:=\min\left\{\frac{1}{4}C_3, \frac{3}{2}A_3\right\}
\end{equation}

\begin{lemma}\label{tem1012-12}
Assume that $0<\beta_{12}<\widehat{\beta}, \beta_{13}<0, \beta_{23}<0$. Then the matrix $M_B[e_1, e_2, e_3]$ is strictly diagonally dominant and the matrix $M_B[e_1, e_2, e_3]$ is positively definite  for $\mu$ large enough. Moreover,
\begin{equation}
\theta_{min}\geq \gamma,
\end{equation}
where $\theta_{min}$ is the minimum eigenvalues of $M_B[e_1, e_2, e_3]$ and $\gamma$ is defined in \eqref{tem1012-7}, independently on $\mu$.
\end{lemma}

\begin{proof}
By \eqref{tem1012-2}, \eqref{tem1012-3} and Lemma \ref{tem1012-1} we have
\begin{align}\label{tem1012-9}
 &\int_{\mathbb{R}^3}|e_1|^{6}  - \sum_{j=2}^3\left|\int_{\mathbb{R}^3}\beta_{1j}|e_1|^{3}|e_j|^{3}\right| \nonumber \\
   & \geq C_3-\beta_{12}\|e_1\|^3_6 \|e_3\|^3_6+|\beta_{13}|\int_{\mathbb{R}^3}|e_1|^{3}|e_3|^{3}\nonumber\\
   & \geq C_3-\beta_{12}C_4^{\frac{1}{2}}(3A_3)^{\frac{1}{2}}+|\beta_{13}|\int_{\mathbb{R}^3}|e_1|^{3}|e_3|^{3}\nonumber\\
   & \geq \frac{1}{4}C_3,  \text{ for } \mu \text{ large enough. }
\end{align}
Similarly, we see that
\begin{equation}\label{tem1012-6}
\int_{\mathbb{R}^3}|e_2|^{6} - \sum_{j\neq 2}\left|\int_{\mathbb{R}^3}\beta_{2j}|e_2|^{3}|e_j|^{3}\right|\geq \frac{1}{4}C_3,  \text{ for } \mu \text{ large enough. }
\end{equation}
It follows from \eqref{tem1012-2} and \eqref{tem1012-3} that
\begin{align}\label{tem1012-10}
 \int_{\mathbb{R}^3}|e_3|^{6}  &- \sum_{j=1}^2\left|\int_{\mathbb{R}^3}\beta_{j3}|e_3|^{3}|e_j|^{3}\right| \nonumber \\
   & = \int_{\mathbb{R}^3}|e_3|^{6}+|\beta_{13}|\int_{\mathbb{R}^3}|e_1|^{3}|e_3|^{3}+
   |\beta_{23}|\int_{\mathbb{R}^3}|e_2|^{3}|e_3|^{3}\\
   & \geq \frac{3}{2}A_3, \text{ for } \mu \text{ large enough. }\nonumber
\end{align}

It follows from \eqref{tem1012-9}, \eqref{tem1012-6} and \eqref{tem1012-10} that the matrix $M_B[e_1, e_2, e_3]$ is strictly diagonally dominant, and so the matrix $M_B[e_1, e_2, e_3]$ is positively definite. For any eigenvalue $\theta$ of $M_B[e_1, e_2, e_3]$, by \eqref{tem1012-9}, \eqref{tem1012-6}, \eqref{tem1012-10} and the Gershgorin circle theorem we know that $\theta\geq \gamma$ and so  $\theta_{min}\geq \gamma$.
\end{proof}

Denote
\begin{equation}\label{tem1012-17}
\widehat{t}^4=\frac{\max\left\{\sum_{i=1}^2\|e_i\|_i^2, 3A_3\right\}}{\gamma}.
\end{equation}
Then we have the following lemma

\begin{lemma}\label{tem1012-20}
Suppose that the assumptions in Lemma \ref{tem1012-12} are satisfied, then there exists $(t_{1,\mu},t_{2,\mu},t_{3,\mu})\in (\R^+)^3$ such that $(t_{1,\mu}e_1,t_{2,\mu}e_2,t_{3,\mu}e_3)\in \mathcal{M}$. Moreover,
\begin{equation}\label{tem1012-14}
0<t_{i,\mu}<\widehat{t}, \quad i=1,2,3,
\end{equation}
where $\widehat{t}$ is defined in \eqref{tem1012-17} and independent on $\mu$.
\end{lemma}

\begin{proof}
For any $t_1,t_2,t_3>0$, consider
\begin{align*}
 f(t_1,t_2,t_3)&:= I(t_1e_1,t_2e_2,t_3e_3) \\
   & =\frac{1}{2}\sum_{i=1}^3t_i^2\|e_i\|_i^2-\frac{1}{6}\sum_{i=1,j}^3t_i^3t_j^3
   \int_{\mathbb{R}^3}\beta_{ij}|e_i|^{3}|e_j|^{3}.
\end{align*}
By Lemma \ref{tem1012-12} we get that
\begin{align*}
   f(t_1,t_2,t_3)& \leq \frac{1}{2}\sum_{i=1}^2t_i^2\|e_i\|_i^2+\frac{3}{2}A_3t_3^2-\frac{1}{6}\gamma\sum_{i=1}^3t_i^6\\
   & \leq \sum_{i=1}^3\left(\frac{1}{2}\max\left\{\sum_{i=1}^2\|e_i\|_i^2, 3A_3\right\}t_i^2-\frac{1}{6}\gamma t_i^6\right),
\end{align*}
which implies that $f(t_1,t_2,t_3)<0$ when $t_1,t_2,t_3>\widehat{t}$, where $\widehat{t}$ is defined in \eqref{tem1012-17}.
Thus,
\begin{equation*}
\max_{t_1,t_2,t_3>0}f(t_1,t_2,t_3)=\max_{0<t_1,t_2,t_3\leq\widehat{t}}f(t_1,t_2,t_3).
\end{equation*}

Based on the above arguments, we know that the function $f(t_1,t_2,t_3)$ has a global maximum point in $\overline{(\R^+)^3}$, and the global maximum point $t_{i,\mu}\leq \widehat{t}, i=1,2,3$. Then following the proof of Lemma \ref{Energyestimates6}, we get that the global maximum point $t_{i,\mu}>0, i=1,2,3$. Thus, the global maximum point $(t_{1,\mu}, t_{2,\mu}, t_{3,\mu})$ is an interior point in $ (\R^+)^3$,  and so $(t_{1,\mu}, t_{2,\mu}, t_{3,\mu})$ is a critical point of $f(t_1,t_2,t_3)$. Therefore, we get that $(t_{1,\mu}e_1,t_{2,\mu}e_2,t_{3,\mu}e_3)\in \mathcal{M}$.
\end{proof}

\begin{proof}[\bf Proof of Proposition \ref{tem1012-19}]
By Lemma \ref{tem1012-20}, we get that $(t_{1,\mu}e_1,t_{2,\mu}e_2,t_{3,\mu}e_3)\in \mathcal{M}$. Note that $\beta_{13},\beta_{23}<0$. Then we see that
\begin{align}\label{tem1012-22}
  A & \leq I(t_{1,\mu}e_1,t_{2,\mu}e_2,t_{3,\mu}e_3)=I_{12}(t_{1,\mu}e_1,t_{2,\mu}e_2)+I_3(t_{3,\mu}e_3) \nonumber\\
   & \quad -\frac{1}{3}\sum_{i=1}^2t_{i,\mu}^3t_{3,\mu}^3\int_{\mathbb{R}^3}\beta_{i3}|e_i|^{3}|e_3|^{3}\nonumber\\
   & \leq \max_{t_1,t_2>0}I_{12}(t_{1}e_1,t_{2}e_2)+\max_{t_3>0}I_3(t_{3}e_3)
   +\frac{1}{3}\sum_{i=1}^2\widehat{t}^6|\beta_{i3}|\int_{\mathbb{R}^3}|e_i|^{3}|e_3|^{3},
\end{align}
where $\widehat{t}$ is defined in \eqref{tem1012-17}. Letting $\mu\rightarrow \infty$ in \eqref{tem1012-22}, by \eqref{tem1012-2} we have
\begin{equation}\label{tem1012-24}
A \leq \max_{t_1,t_2>0}I_{12}(t_{1}e_1,t_{2}e_2)+\max_{t_3>0}I_3(t_{3}e_3).
\end{equation}
It is easy to see that $\max_{t_3>0}I_3(t_{3}e_3)=A_3$. By \eqref{tab1122-1-4} we have
\begin{equation*}
\max_{t_1,t_2>0}I_{12}(t_{1}e_1,t_{2}e_2)=A_{12}.
\end{equation*}
Combining this with \eqref{tem1012-24} that we know that
\begin{equation*}
A \leq A_{12}+A_3.
\end{equation*}
\end{proof}

\begin{proof}[\bf The proof of Theorem \ref{thm3}-(2)]
Based on Lemma \ref{tem1012-25} and Proposition \ref{tem1012-19}, the proof is same as proof of Theorem \ref{thm3}-(1), and so we omit it.
\end{proof}

\subsection{The case for $N\geq 5$}

In this subsection, we assume that $\beta_{12}>0, \beta_{13}<0, \beta_{23}<0$ and give the proof of Theorem \ref{thm4}. Set $p=\frac{N}{N-2}$. Note that $N\geq 5$, $1<p<2$.
Define
\begin{align*}
\mathcal{N}_{12}:=\Bigg\{(u_1,u_2)\neq \mathbf{0}:
 \sum_{i=1}^2\|u_i\|_i^2=\sum_{i,j=1}^2\int_{\mathbb{R}^N}\beta_{ij}|u_i|^{p}|u_j|^{p} \Bigg \},
\end{align*}
and
\begin{equation*}
\widehat{A}_{12}:=\inf_{\mathbf{u}\in \mathcal{N}_{12}}I_{12}(u_1,u_2),
\end{equation*}
where $\beta_{ii}=1$ and $I_{12}$ is defined in \eqref{1.10-4}. By \cite[Theorem 1.1]{Chen-Zou2015} we get that $\widehat{A}_{12}$ is attained and $\widehat{A}_{12}=I_{12}(\omega^{12}_1,\omega^{12}_2)$. Consider
\begin{align*}
\widetilde{\mathcal{N}}_{12}:=\Bigg\{(u_1,u_2)\neq \mathbf{0}:
 \sum_{i=1}^2\|u_i\|_i^2\leq\sum_{i,j=1}^2\int_{\mathbb{R}^N}\beta_{ij}|u_i|^{p}|u_j|^{p} \Bigg \},
\end{align*}
and
$$
\widetilde{A}_{12}:=\inf_{(u_1,u_2)\in \widetilde{\mathcal{N}}_{12}}\frac{1}{N}\sum_{i=1}^2\|u_i\|_i^2.
$$
Then we have the following lemma
\begin{lemma}
Assume that $N\geq 5$, then $\widetilde{A}_{12}=\widehat{A}_{12}$.
\end{lemma}
\begin{proof}
The proof is similar to that of Lemma \ref{tem512-1}, and so we omit it.
\end{proof}

\begin{lemma}\label{tem1013-25}
Assume that $\beta_{12}>0, \beta_{13}<0, \beta_{23}<0$, then $\widetilde{A}\geq A_3+\widehat{A}_{12}$, where $\widetilde{A}$ is defined in \eqref{tem512-5}.
\end{lemma}
\begin{proof}
The proof is same as Lemma \ref{tem106-2}, and so we omit it.
\end{proof}

\begin{prop}\label{tem1013-24}
Assume that $\beta_{12}>0, \beta_{13}<0, \beta_{23}<0$, then $ \widetilde{A}\leq A_3+\widehat{A}_{12}$, where $\widetilde{A}$ is defined in \eqref{tem512-5}.
\end{prop}

\medbreak

Next, we present the proof of Proposition \ref{tem1013-24}. Note that $\widehat{A}_{12}=I_{12}(\omega^{12}_1,\omega^{12}_2)$. Since $z^3_\mu\rightharpoonup 0$ in $D^{12}(\mathbb{R}^N)$ as $\mu\rightarrow \infty$, then $(z^3_\mu)^p\rightharpoonup 0$ in $L^{2}(\mathbb{R}^N)$ as $\mu\rightarrow \infty$. Thus
\begin{equation}\label{tem1013-9}
\int_{\mathbb{R}^N}|z^3_\mu|^{p}|\omega^{12}_i|^{p}\rightarrow 0 \text{ as } \mu\rightarrow \infty, ~\text{ for any } i=1,2.
\end{equation}
Set $z_1:= \omega^{12}_1, z_2:= \omega^{12}_2, z_3:= z^3_\mu$.
Then we have
\begin{equation}\label{tem1013-10}
\int_{\mathbb{R}^N}\left(|\nabla z_3|^{2}-\frac{\lambda_{3}}{|x|^2} z_3^{2}\right)=\int_{\mathbb{R}^N}|z_3|^{2p}=NA_3,
\end{equation}
and
\begin{equation}\label{tem1013-12}
\int_{\mathbb{R}^N}|z_i|^{p}|z_3|^{p}\rightarrow 0 \text{ as } \mu\rightarrow \infty, ~\text{ for any } i=1,2.
\end{equation}

\begin{lemma}\label{tem1013-13}
Suppose that $\beta_{12}>0, \beta_{13}<0, \beta_{23}<0$, then there exists $(t_{\mu},s_{\mu})\in (\R^+)^2$ such that $(t_{\mu}z_1,t_{\mu}z_2,s_{\mu}z_3)\in \mathcal{N}_{12,3}$ when $\mu$ is large enough. Moreover,
\begin{equation}\label{tem1012-14}
\lim_{\mu\rightarrow\infty}t_{\mu}=\lim_{\mu\rightarrow\infty}s_{\mu}=1.
\end{equation}
\end{lemma}

\begin{proof}
For simplicity, we denote
\begin{equation*}
B_{\mu}:=|\beta_{13}|\int_{\mathbb{R}^N}|z_1|^{p}|z_3|^{p}+|\beta_{23}|\int_{\mathbb{R}^N}|z_2|^{p}|z_3|^{p}.
\end{equation*}
It is easy to see that $B_{\mu}>0$. By \eqref{tem1013-12} we have
\begin{equation}\label{tem1013-17}
\lim_{\mu\rightarrow\infty}B_{\mu}=0.
\end{equation}
It is standard to see that $(tz_1, tz_2, sz_3) \in \mathcal{N}_{12,3}$ for $t,s>0$ is equivalent to $t,s>0$ satisfying
\begin{equation}\label{tem1013-21}
NA_{12}t^2=NA_{12}t^{2p}-B_{\mu}t^ps^p, \quad NA_{3}s^2=NA_{3}s^{2p}-B_{\mu}t^ps^p,
\end{equation}
that is,
\begin{equation}\label{tem1013-16}
NA_{12}t^{2-p}=NA_{12}t^{p}-B_{\mu}s^p, \quad NA_{3}s^{2-p}=NA_{3}s^{p}-B_{\mu}t^p.
\end{equation}
Note that the first equation in \eqref{tem1013-16} is equivalent to
\begin{equation}
s=h(t):=\left(\frac{NA_{12}t^{p}-NA_{12}t^{2-p}}{B_{\mu}}\right)^{\frac{1}{p}}>0, \quad
t>1.
\end{equation}
Thus, it just needs to prove that
\begin{equation*}
g(t):= \left(\frac{NA_{12}-NA_{12}t^{2-2^*}}{B_{\mu}}\right)^{\frac{N-4}{N}}+
t^{2^*-2}\left(\frac{B_{\mu}^2-N^2A_3A_{12}+N^2A_3A_{12}t^{2-2^*}}{B_{\mu}}\right),
\end{equation*}
has a solution $t>1$. It is easy to see that $\lim_{t\searrow 1} g(t)=B_{\mu}>0$. We deduce from \eqref{tem1013-17} that
$B_{\mu}^2-N^2A_3A_{12}<0$ for $\mu$ large enough, then we get that $\lim_{t\rightarrow+\infty} g(t)=-\infty$. Hence, there exists $t_\mu>1$ such that $g(t_\mu)=0$. Set $s_\mu=h(t_\mu)$, then we have $s_\mu>0$. Thus, $(t_{\mu}z_1,t_{\mu}z_2,s_{\mu}z_3)\in \mathcal{N}_{12,3}$ for $\mu$ large enough.

Next, we claim that $t_{\mu}, s_{\mu}$ are uniformly bounded. By contradiction, up to a subsequence, we assume that $t_{\mu}\rightarrow +\infty$ as $\mu\rightarrow +\infty$. By \eqref{tem1013-21} we have
\begin{equation*}
NA_{12}t_\mu^{2p}-NA_{12}t_\mu^2=NA_{3}s_\mu^{2p}-NA_{3}s_\mu^2,
\end{equation*}
which yields that $s_{\mu}\rightarrow +\infty$ as $\mu\rightarrow +\infty$. It follows from  \eqref{tem1013-21} that
\begin{equation*}
\lim_{\mu\rightarrow\infty}B_{\mu}^2=\lim_{\mu\rightarrow\infty}\frac{N^2A_{12}A_{3}(t_\mu^{2p}-t_\mu^2)(s_\mu^{2p}-s_\mu^2)}{t_\mu^{2p}s_\mu^{2p}}
=N^2A_{12}A_{3}>0,
\end{equation*}
which contradicts with \eqref{tem1013-17}. Thus, $t_{\mu}, s_{\mu}$ are uniformly bounded. Passing to a subsequence, by \eqref{tem1013-17} and \eqref{tem1013-21} we have
$\lim_{\mu\rightarrow\infty}t_{\mu}=\lim_{\mu\rightarrow\infty}s_{\mu}=1$.

\end{proof}

\begin{proof}[\bf Proof of Proposition \ref{tem1013-24}]
By Lemma \ref{tem1013-13} we know that $(t_{\mu}z_1,t_{\mu}z_2,s_{\mu}z_3)\in \mathcal{N}_{12,3}$, thus
\begin{align}\label{tem1013-22}
  \widetilde{A} & \leq I(t_{\mu}z_1,t_{\mu}z_2,s_{\mu}z_3)=I_{12}(t_{\mu}z_1,t_{\mu}z_2)+I_3(s_{\mu}z_3)\nonumber\\
   & \quad -\frac{1}{p}\sum_{i=1}^2t_{\mu}^ps_{\mu}^p\int_{\mathbb{R}^N}\beta_{i3}|z_i|^{p}|z_3|^{p}.
\end{align}
Letting $\mu\rightarrow\infty$ in \eqref{tem1013-22}, by \eqref{tem1013-12} and Lemma \ref{tem1013-13} we have $\widetilde{A} \leq A_3+\widehat{A}_{12}$.
\end{proof}

\begin{proof}[\bf The proof of Theorem \ref{thm4}]
Based on Lemma \ref{tem1013-25} and Proposition \ref{tem1013-24}, the proof is same as proof of Theorem \ref{thm3}-(1), and so we omit it.
\end{proof}

\noindent{\bf Acknowledgements}\,\, The authors join to thank \textit{Hugo Tavares} for several fruitful discussions and valuable comments. Song You was supported by the Scientific Research Foundation of Chongqing Normal University (Grant No. 22XLB031), the Natural Science Foundation of Chongqing (Grant No. CSTB2023NSCQ-MSX0467), the Science and Technology Research Program of Chongqing Municipal Education Commission (Grant No. KJQN202300520) and NSFC (No.12301130). Jianjun Zhang was supported by NSFC (No.12371109,11871123).

\vskip0.1in
\noindent{\bf Data Availability Statement}\,\, Data is contained within the article.

\vskip0.1in
\noindent{\bf Conflicts of Interest}\,\, On behalf of all authors, the corresponding author states that there is no conflict of interest.

\end{document}